\documentclass[EJP]{ejpecp} 

\SHORTTITLE{Small Ball Probabilities for SPDEs}

\TITLE{Small Ball Probabilities for the Fractional Stochastic Heat Equation Driven by a Colored Noise}

\AUTHORS{
  Jiaming Chen\footnote{Dept. of Mathematics, University of Rochester, Rochester, NY 14627,
    \EMAIL{jchen1994@urgrad.rochester.edu}}}
  

\AMSSUBJ{60H15} 
\AMSSUBJSECONDARY{60G60} 

\SUBMITTED{January 6, 2023} 
\ACCEPTED{February 11, 2025} 

\VOLUME{0}
\YEAR{2020}
\PAPERNUM{0}
\DOI{10.1214/YY-TN}

\ABSTRACT{We consider the fractional stochastic heat equation on the $d$-dimensional torus $\mathbb{T}^d:=\left[-\frac{1}{2},\frac{1}{2}\right]^d$, $d\geq 1$, with periodic boundary conditions:
\[
\partial_t u(t,\textbf{x})= -(-\Delta)^{\alpha/2}u(t,\textbf{x})+\sigma(t,\textbf{x},u)\dot{F}(t,\textbf{x})\quad \textbf{x}\in \mathbb{T}^d,t\in\mathbb{R}_+
,\]
where $\alpha\in(1,2]$ and $\dot{F}(t,\textbf{x})$ is a generalized Gaussian noise which is white in time and colored in space. Assuming that $\sigma$ is Lipschitz in $u$ and uniformly bounded, we estimate small ball probabilities for the solution $u$ when $u(0,\textbf{x})\equiv 0$.}

\newcommand{\R}{\mathbb{R}}
\newcommand{\Z}{\mathbb{Z}}

\newcommand{\N}{\mathbb{N}}

\newcommand{\E}{\mathbb{E}}

\begin{document}



\section{Introduction}
Stochastic Partial Differential Equations (SPDEs) have gathered significant attention in recent years, due to their broad applications across diverse fields such as applied mathematics, theoretical physics, polymer science, and mathematical finance. For a comprehensive introduction in this domain, see \cite{khoshnevisan2014analysis}. Meanwhile, fractional SPDEs driven by homogeneous noise have also emerged as a vibrant area of research, attracting substantial interest due to their rich mathematical structure and wide-ranging applicability. See for example \cite{foondun2013stochastic}, \cite{balan2014note}, \cite{liu2017some}, \cite{nualart2018moment} and \cite{guerngar2020moment}.

In this paper, we estimate small ball probabilities for the solution to the fractional stochastic heat equation of the type:
\begin{equation}
\label{SHE}
\partial_t u(t,\textbf{x})=-(-\Delta)^{\alpha/2}u(t,\textbf{x})+\sigma(t,\textbf{x},u)\dot{F}(t,\textbf{x}),\quad \textbf{x}\in \mathbb{T}^d,t\in\R_+,
\end{equation}
with given initial deterministic profile $u(0,\cdot)=u_0:\mathbb{T}^d\to \R$ where $\mathbb{T}^d:=\left[-\frac{1}{2},\frac{1}{2}\right]^d$ is a $d$-dimensional torus. The operator $-(-\Delta)^{\alpha/2}$, where $1<\alpha\leq 2$, is the fractional power Laplacian on $\mathbb{T}^d$. The centered Gaussian noise $\dot{F}$ is white in time and homogeneous in space, i.e.,
\begin{equation*}
\mathbb{E}\left[\dot{F}(t,\textbf{x})\dot{F}(s,\textbf{y})\right]=\delta_0(t-s)\Lambda(\textbf{x}-\textbf{y}),
\end{equation*}
where $\delta_0$ is the Dirac delta generalized function and $\Lambda:\mathbb{T}^d\to\R_+$ is a non-negative generalized function whose Fourier series is given by
\begin{equation}
\label{lambdafourier}
\Lambda(\textbf{x})=\sum_{\textbf{n}\in\mathbb{Z}^d}\lambda(\textbf{n})\exp(2\pi i \textbf{n}\cdot\textbf{x}),
\end{equation}
where $\textbf{n}\cdot\textbf{x}$ represents the dot product of two $d$-dimensional vectors. 

We are working on the torus $\mathbb{T}^d$, so that the expression $\textbf{x}-\textbf{y}$ should be interpreted as the unique point $\textbf{z}$ in $\left[-\frac{1}{2},\frac{1}{2}\right]^d$ such that $\textbf{x}-\textbf{y} = \textbf{z}+\textbf{n}$ for some $\textbf{n}\in\mathbb{Z}^d$. Moreover, we will need the following two assumptions on the function $\sigma:\mathbb{R}_+\times\mathbb{T}^d\times\R\to \R$.

\begin{hypothesis}
There exists a constant $\mathcal{D}>0$ such that for all $t\geq 0$, $\textbf{x}\in\mathbb{T}^d$ and $u,v\in\R$,
\begin{equation}
\label{hypothesis1}
\vert \sigma(t,\textbf{x},u)-\sigma(t,\textbf{x},v)\vert\leq\mathcal{D}|u-v|.
\end{equation}
\end{hypothesis}

\begin{hypothesis}
There exist constants $\mathcal{C}_1$, $\mathcal{C}_2>0$ such that for all $t\geq 0$, $\textbf{x}\in\mathbb{T}^d$ and $u\in\R$,
\begin{equation}
\label{hypothesis2}
\mathcal{C}_1\leq \sigma(t,\textbf{x},u)\leq \mathcal{C}_2.
\end{equation}
\end{hypothesis}

Assumption \eqref{hypothesis2} assures that $\sigma$ is non-zero, which is helpful for us to control the variance of noise terms.  We say that a process $\{u(t,\textbf{x})\}_{t\geq 0, \textbf{x}\in\mathbb{T}^d}$ on a probability space $(\Omega,\mathcal{F},P)$ is a \textit{mild solution} to \eqref{SHE} in the sense of Walsh \cite{walsh1986anintroductiontostochastic} if it is square integrable, adapted with respect to the filtration induced by $F$ and satisfies:
\begin{equation}
\label{mild}
u(t,\textbf{x})=\int_{\mathbb{T}^d}\bar{p}(t,\textbf{x}-\textbf{y})u_0(\textbf{y})d\textbf{y}+\int_{[0,t]\times\mathbb{T}^d}\bar{p}(t-s,\textbf{x}-\textbf{y})\sigma(s,\textbf{y},u(s,\textbf{y}))F(dsd\textbf{y}),
\end{equation}
where $\bar{p}:\R_+\times\mathbb{T}^d\to \R_+$ is the fundamental solution of the fractional heat equation on $\mathbb{T}^d$,
\begin{equation}
\label{equation1}
\begin{cases}
\partial_t\bar{p}(t,\textbf{x})=-(-\Delta)^{\alpha/2}\bar{p}(t,\textbf{x}),\\
\bar{p}(0,\textbf{x})=\delta_0(\textbf{x}).
\end{cases}
\end{equation}

We give more information about the heat kernel in Section \ref{estimates}. Under the Lipschitz assumption \eqref{hypothesis1} and given an initial profile $u_0(\textbf{x})=0$, it is well known (see \cite{dalang1999extending} and Appendix of \cite{liu2017some}) that if the Fourier coefficients of $\Lambda(\textbf{x})$ satisfy
\begin{equation}
\label{Lambda}
\sum_{\textbf{n}\in\mathbb{Z}^d}\frac{\mathcal{F}(\Lambda)(\textbf{n})}{1+\vert \textbf{n}\vert^\alpha}<\infty,
\end{equation}
where $\vert \cdot\vert$ denotes the Euclidean norm, then there exists a unique random field solution $u(t,\textbf{x})$ to equation \eqref{mild}. Examples of spatial correlation function satisfying \eqref{Lambda} are:
\begin{enumerate}
\item[1.] The spatial covariance function that is an analogue of the Riesz kernel on $\R^d$, which is specified in \cite{chen2023parabolic}. Moreover, \cite{brosamler1983laws}, Lemma 2.9, proved that there exists a constant $C(d,\beta)>0$ such that the covariance function admits the following estimate,
\begin{equation}
\label{Lambdabound}
\Lambda(\textbf{x})\leq C(d,\beta)\vert \textbf{x}\vert^{-\beta},
\end{equation}
when $0<\beta<d$. There are two positive constants $c_1(d)$ and $c_2(d)$ such that the Fourier coefficients for $\Lambda$ are given by
\begin{equation}
\label{lambdabound}
\lambda(\textbf{n})=\begin{cases}
c_1(d) & \text{if~$\textbf{n}=\textbf{0}$};\\
c_2(d)\vert \textbf{n}\vert^{-(d-\beta)}& \text{if~$\textbf{n}\in\mathbb{Z}^d\setminus \{\textbf{0}\}$},
\end{cases}
\end{equation}
and it is easy to check that condition \eqref{Lambda} holds whenever $\beta<\alpha$.
\item[2.] Space-time white noise $\Lambda(\textbf{x})=\delta_0(\textbf{x})$. In this case, $\lambda(\textbf{n})$ is a constant and \eqref{Lambda} is only satisfied when $\alpha>d$, that is, $d=1$ and $1<\alpha\leq 2$.
\end{enumerate}
We refer readers to \cite{foondun2013stochastic} for more examples.

Small ball problems have been well studied and one can see \cite{li2001gaussian} for an overview of known results on Gaussian processes and references on other processes. In summary, we investigate the behavior of the probability that a stochastic process $X_t$ starting at the origin stays near its initial location for a long time, i.e.,
\[\log P\left(\sup_{0\leq t\leq T}\vert X_t\vert<\varepsilon\right)\]
as $\varepsilon\downarrow 0$. When $X_t$ is Brownian motion, small ball estimates are obtained from the reflection principle or the study of eigenvalues, among other techniques, resulting in a sharp approximation, as $\varepsilon\downarrow 0$
\[\log P\left(\sup_{0\leq t\leq T}\vert X_t\vert<\varepsilon\right)\sim-\frac{\pi^2T}{8\varepsilon^2}.\]

Often these results are given in terms of metric entropy estimates \cite{kuelbs1993metric}, which are hard to explicitly compute. However, there are a few Gaussian measures for which small ball probabilities can be completely determined. One exceptional case is Brownian sheet, for which fairly sharp small ball estimates are explicitly known, see \cite{bass1988probability} and \cite{talagrand1994small}. \cite{khoshnevisan1998chung} studied small ball estimates for the integrated Brownian motion using local time techniques. In \cite{dalang2009minicourse}, the authors provided small ball estimates for Gaussian processes that satisfy a certain condition which is related to the Gaussian concept of local nondeterminism.

However, there have not been many explorations of small ball probabilities in the context of SPDEs. \cite{lototsky2017small} has studied small ball problems for a linear SPDE with additive white noise, where the solution is a Gaussian process. \cite{martin2003small} followed the approach of \cite{talagrand1994small} to study a stochastic wave equation
\begin{equation}
\label{waveequation}
\partial^2_tu(t,x)=\frac{1}{2}\partial^2_xu(t,x)+f(u)+g(t,x)\dot{W}(t,x),
\end{equation}
where $\dot{W}(t,x)$ is a space-time white noise and $g(t,x)$ is a deterministic function. Without $f(u)$, the solution $u$ would be a Gaussian process of the type studied by \cite{bass1988probability} and \cite{talagrand1994small}. 

Although the stochastic wave equation \eqref{waveequation} includes a multiplicative noise term, in our case, the noise coefficient $\sigma$ depends on the solution $u$. This dependency causes $u$ to deviate from being Gaussian. The study of small ball probabilities for equations of this type has seen notable developments: \cite{athreya2021small} analyzed the stochastic heat equation driven by space-time white noise; \cite{foondun2022small} extended this work by considering small ball probabilities under a H\"older semi-norm; and \cite{han2024support} derived results under the assumption that the diffusion coefficient $\sigma$ is uniformly elliptic but only H\"older continuous in $u$. All these works were restricted to the one-dimensional torus due to the limitations imposed by space-time white noise.

Recently, \cite{chen2024small} adapted the approach of \cite{athreya2021small}, incorporating the foundational framework from \cite{dalang1999extending}, to investigate small ball probabilities for the stochastic heat equation driven by spatially homogeneous noise on the one-dimensional torus. The analysis utilized the Da Prato-Kwapien-Zapczyk factorization method to estimate fluctuations in the noise term and introduced a distinct time step size compared to \cite{athreya2021small}. The primary aim of the present work is to generalize these results to the fractional stochastic heat equation driven by spatially homogeneous noise on the $d$-dimensional torus. 

\section{Main Result}
\begin{theorem}
\label{thm}
Under the conditions \eqref{hypothesis1} and \eqref{hypothesis2}, if $u(t,\textbf{x})$ is the solution to \eqref{SHE} with $u_0(\textbf{x})\equiv 0$ and $\Lambda(\textbf{x})$ specified in \eqref{Lambdabound}, then there are positive constants $\textbf{C}_0,\textbf{C}_1,\textbf{C}_2,\textbf{C}_3,\mathcal{D}_0$ and $\varepsilon_0$ depending entirely on $\mathcal{C}_1,\mathcal{C}_2$, $\alpha$, $\beta$ and $d$, such that for all $\mathcal{D}<\mathcal{D}_0,0<\varepsilon<\varepsilon_0$ and $T>1$, we have
\begin{enumerate}
\item[(a)] when $2\beta\leq \alpha$ and $d=1$, 
\begin{equation*}
\textbf{C}_0\exp\left(-\frac{\textbf{C}_1T}{\varepsilon^{\frac{4\alpha -2\beta}{\alpha-\beta}}}\right)\leq P\left(\sup_{\substack{0\leq t\leq T\\ \textbf{x}\in\mathbb{T}}}\vert u(t,\textbf{x})\vert\leq \varepsilon\right)\leq\textbf{C}_2\exp\left(-\frac{\textbf{C}_3T}{\varepsilon^{\frac{2(\alpha+\beta)}{\alpha-\beta}}}\right);
\end{equation*}
\item[(b)] in other cases, 
\begin{equation*}
0<P\left(\sup_{\substack{0\leq t\leq T\\ \textbf{x}\in\mathbb{T}^d}}\vert u(t,\textbf{x})\vert\leq \varepsilon\right)\leq\textbf{C}_2\exp\left(-\frac{\textbf{C}_3T}{\varepsilon^{\left(\frac{4\alpha-2\beta}{\alpha-\beta}\right)\wedge\left(\frac{2(\alpha d+\beta)}{d(\alpha-\beta)}\right)}}\right).
\end{equation*}
\end{enumerate}
\end{theorem}

Here we make a couple of remarks. These could be of independent interests.
\begin{remark}
\begin{enumerate}
\item[(a)] When $\alpha=2\beta$ and $d=1$, both lower and upper bounds include a term $\varepsilon^{-6}$ in the exponent. However, there is a discrepancy between the exponents of $\varepsilon$ in other cases due to the fact that $c_0$ depends on $\varepsilon$, as defined in \eqref{c0}.
\item[(b)] If $\beta<\alpha<2\beta$ or $d>1$, we have a negative lower bound of $P(E_0)$ in \eqref{E0a} and \eqref{E0b}, which gives a trivial lower estimation in Theorem \ref{thm}(b). 
\end{enumerate}
\end{remark}

We outline the key steps in the proof of Theorem \ref{thm}. In \cite{athreya2021small}, the authors decomposed a one-dimensional spatial interval into small subintervals. Extending this approach to the $d$-dimensional setting, we instead partition the domain into a sequence of small boxes.

Concerning the temporal coefficient $c_0$ introduced in \cite{athreya2021small}, it was shown there to be uniformly bounded, leveraging the exponential decay of spatial correlations as the distance between two points increases. However, for spatially homogeneous noise, the decay of correlation is significantly slower. As a result, a uniform bound on $c_0$ is no longer feasible. Instead, we control $c_0$ in terms of $\varepsilon$, as specified in \eqref{c0}.

To establish the upper bound in Theorem \ref{thm}, we employ a perturbation argument that approximates the solution $u$ by a Gaussian random field within small regions. This is complemented by deriving bounds for a corresponding Gaussian process. For the lower bound, we utilize the Gaussian correlation inequality along with a change-of-measure technique inspired by \cite{athreya2021small}. Additionally, we demonstrate that the approximation error can be effectively controlled by appropriately choosing the time intervals over which the coefficients are frozen. A critical component of our analysis is the Markov property of \eqref{SHE} with respect to the time variable $t$. This property allows us to reduce the problem to studying the behavior of the non-Gaussian solution over small time intervals.

The remainder of the paper is organized as follows: In section \ref{keyprop}, we present Proposition \ref{prop} and its connection to Theorem \ref{thm}. Section \ref{estimates} provides several essential estimates, while section \ref{proofprop} contains the proof of Proposition \ref{prop}, thereby completing the paper.

Throughout this work, $C$ and $C'$ denote positive constants whose values may change between lines. Dependence of constants on specific parameters is explicitly indicated by listing the parameters in parentheses.

\section{Key Proposition}\label{keyprop}
We decompose $\left[-\frac{1}{2},\frac{1}{2}\right]^d$ into subintervals of length $\varepsilon^2$ in each coordinate, and divide $[0,T]$ into intervals of length $c_0\varepsilon^4$ where $c_0$ denotes the temporal coefficient satisfying 
\begin{equation}
\label{c0}
c_0=\mathcal{C}_0\varepsilon^{\frac{2\alpha d-4\beta}{\beta}},
\end{equation}
where $0<\mathcal{C}_0<\mathcal{C}$ and $\mathcal{C}$ is specified in the proof of Lemma \ref{coeffbound}. 
\begin{remark}
Unlike the space-time white noise case in \cite{athreya2021small}, $c_0$ needs to be selected depending on $\varepsilon$ due to the large correlation between pairs of points. $c_0$ does not appear in either the upper bound nor lower bound for small ball probabilities. 
\end{remark}

We define $t_i=ic_0\varepsilon^4,x_{j}=j\varepsilon^2$ where $i\in\N$ and $j\in\Z$, and
\begin{equation*}
n_1:=\min\left\lbrace n\in\Z:n\varepsilon^2>\frac{1}{2}\right\rbrace.
\end{equation*}
It is clear that $x_{n_1}>\frac{1}{2}$ and $x_{n_1-1}\leq \frac{1}{2}$, therefore, by symmetry, the point $(x_{j_1},x_{j_2},...,x_{j_d})$ lies in $\left[-\frac{1}{2},\frac{1}{2}\right]^d$ if
\begin{equation}
\label{jbound}
-n_1+1\leq j_k\leq n_1-1 \text{~for $k=1,2,...,d$}.
\end{equation}

We now consider a sequence of sets $R_{i,j}\subset \R\times\R^d$ as
\begin{equation}
\label{Pij}
R_{i,j}=\left\lbrace(t_i,x_{j_1},x_{j_2},...x_{j_d})\vert -n_1+1\leq j_k\leq j \text{~for~} k=1,2,...,d\right\rbrace,
\end{equation}
so that $R_{i,j}$ represents the set including all grid points within the box $\left[-\frac{1}{2},j\varepsilon^2\right]^d$ at time $t_i$. Below we have two important definitions using this set.

For $n\geq 0$, we define the following sequence of events to determine the upper bound in Theorem \ref{thm},
\begin{equation}
\label{Fn}
F_{n}=\left\lbrace\vert u(t,\textbf{x})\vert\leq t_1^{\frac{\alpha-\beta}{2\alpha}}\text{~for all $(t,\textbf{x})\in R_{n,n_1-1}$}\right\rbrace,
\end{equation}
which is the event that the magnitude of the solution is small (recall that $t_1=c_0\varepsilon^4$) at time $t_n$ for each grid point within $\left[-\frac{1}{2},\frac{1}{2}\right]^d$.

In addition, let $E_{-1}=\Omega$. For $n\geq 0$, we define the following sequence of events to determine the lower bound in Theorem \ref{thm},
\begin{equation}
\label{En}
E_{n}=\left\lbrace\vert u(t_{n+1},\textbf{x})\vert\leq \frac{\mathcal{C}_3}{3}t_1^{\frac{\alpha-\beta}{2\alpha}}\text{,~and~} \vert u(t,\textbf{x})\vert\leq \mathcal{C}_3t_1^{\frac{\alpha-\beta}{2\alpha}}\text{~for all~} t\in[t_n,t_{n+1}],\textbf{x}\in\mathbb{T}^d\right\rbrace,
\end{equation}
which is the event that the magnitude of the solution is small at any time in $[t_n,t_{n+1}]$ for any points within $[-\frac{1}{2},\frac{1}{2}]^d$ and even smaller at time $t_{n+1}$. $\mathcal{C}_3$ is a positive constant such that
\begin{equation}
\label{c3}
\mathcal{C}_3>6\mathcal{C}_2\sqrt{\frac{\ln C_5}{C_6}},
\end{equation}
where $\mathcal{C}_2$ is the uniform bound of $\sigma$ in \eqref{hypothesis2} and $C_5,C_6$ are positive constants defined in Lemma \eqref{larged}.

\begin{proposition}
\label{prop}
Consider the solution to \eqref{SHE} with $u_0(\textbf{x})\equiv 0$, $0<2\beta\leq\alpha$ and $d=1$. Then, there exist positive constants $\textbf{C}_4,\textbf{C}_5,\textbf{C}_6,\textbf{C}_7,\mathcal{D}_0$ and $\varepsilon_1$ depending solely on $\mathcal{C}_1$, $\mathcal{C}_2$, $\alpha$, $\beta$, and $d$ such that for any $0<\varepsilon<\varepsilon_1$ and $\mathcal{D}<\mathcal{D}_0$ in \eqref{hypothesis1}, we have
\begin{enumerate}
\item[(a)]
\begin{equation*}
P\left(F_n\bigg| \bigcap_{k=0}^{n-1}F_k\right)\leq\textbf{C}_4\exp\left(-\frac{\textbf{C}_5}{t_1^{\frac{\beta}{\alpha}}}\right),
\end{equation*}
\item[(b)] 
\begin{equation*}
P\left(E_n\bigg| \bigcap_{k=-1}^{n-1}E_k\right)\geq \textbf{C}_6\exp\left(-\frac{\textbf{C}_7}{t_1^{1-\beta/\alpha}}\right).
\end{equation*}
\end{enumerate}
\end{proposition}
We then demonstrate how Theorem \ref{thm}(a) results from Proposition \ref{prop}.

\textbf{Proof of Theorem \ref{thm}(a)}
\begin{proof}
The event $F_{n}$ in $\eqref{Fn}$ deals with the behavior of $u(t,\textbf{x})$ at time $t_n$, thus combining these events together indicates
\[F:=\bigcap_{n=0}^{\left\lfloor\frac{T}{t_1}\right\rfloor}F_{n}\supset \left\lbrace\vert u(t,\textbf{x})\vert\leq t_1^{\frac{\alpha-\beta}{2\alpha}}, t\in[0,T],\textbf{x}\in\mathbb{T}^d\right\rbrace,\]
and
\[P(F)=P\left(\bigcap_{n=0}^{\left\lfloor\frac{T}{t_1}\right\rfloor}F_{n}\right)=P(F_{0})\prod_{n=1}^{\left\lfloor\frac{T}{t_1}\right\rfloor}P\left(F_{n}\bigg| \bigcap_{k=0}^{n-1}F_{k}\right).\]

With $u_0(\textbf{x})\equiv 0$, $F_0=\Omega$, and Proposition \ref{prop}(a) immediately yields 
\begin{align*}
P(F)&\leq \left[\textbf{C}_4\exp\left(-\frac{\textbf{C}_5}{t_1^{\frac{\beta}{\alpha}}}\right)\right]^{\left\lfloor\frac{T}{t_1}\right\rfloor}\leq \textbf{C}_2\exp\left(-\frac{\textbf{C}_3T}{t_1^{\frac{\alpha+\beta}{\alpha}}}\right).
\end{align*}
Therefore we have
\[
P\left(\left\lbrace\vert u(t,\textbf{x})\vert\leq t_1^\frac{\alpha-\beta}{2\alpha}, t\in[0,T],\textbf{x}\in\mathbb{T}^d\right\rbrace\right)\leq\textbf{C}_2\exp\left(-\frac{\textbf{C}_3T}{t_1^{\frac{\alpha+\beta}{\alpha}}}\right),
\]
and replacing $t_1^\frac{\alpha-\beta}{2\alpha}$ with $\varepsilon_2$ gives
\begin{equation}
\label{upper2}
P\left(\left\lbrace\vert u(t,\textbf{x})\vert\leq \varepsilon_2, t\in[0,T],\textbf{x}\in\mathbb{T}^d\right\rbrace\right)\leq\textbf{C}_2\exp\left(-\frac{\textbf{C}_3T}{\varepsilon_2^{\frac{2(\alpha+\beta)}{(\alpha-\beta)}}}\right).
\end{equation}

We derive that, from \eqref{c0},
\[
\varepsilon_2=t_1^\frac{\alpha-\beta}{2\alpha}=\left(c_0\varepsilon^4\right)^\frac{\alpha-\beta}{2\alpha}=\left(\mathcal{C}_0\varepsilon^{\frac{2\alpha d}{\beta}}\right)^\frac{\alpha-\beta}{2\alpha}<\left(\mathcal{C}_0\varepsilon_1^{\frac{2\alpha d}{\beta}}\right)^\frac{\alpha-\beta}{2\alpha},
\]
so \eqref{upper2} provides the upper bound in Theorem \ref{thm}(a) whenever $\varepsilon_0<\left(\mathcal{C}_0\varepsilon_1^{\frac{2\alpha d}{\beta}}\right)^\frac{\alpha-\beta}{2\alpha}$. 

On the other hand, the event $E_{n}$ in $\eqref{En}$ deals with the behavior of $u(t,\textbf{x})$ at any time between $t_n$ and $t_{n+1}$, thus combining these events together indicates
\[E:=\bigcap_{n=-1}^{\left\lfloor\frac{T}{t_1}\right\rfloor}E_{n}\subset \left\lbrace\vert u(t,\textbf{x})\vert\leq \mathcal{C}_3t_1^{\frac{\alpha-\beta}{2\alpha}}, t\in[0,T],\textbf{x}\in\mathbb{T}^d\right\rbrace,\]
and
\[P(E)=P\left(\bigcap_{n=-1}^{\left\lfloor\frac{T}{t_1}\right\rfloor}E_{n}\right)=P(E_{-1})\prod_{n=0}^{\left\lfloor\frac{T}{t_1}\right\rfloor}P\left(E_{n}\bigg| \bigcap_{k=-1}^{n-1}E_{k}\right).\]
With $u_0(\textbf{x})\equiv 0$ and $E_{-1}=\Omega$, Proposition \ref{prop}(b) immediately yields
\[P(E)\geq \left[\textbf{C}_6\exp\left(-\frac{\textbf{C}_7}{t_1^{1-\beta/\alpha}}\right)\right]^{\left\lfloor\frac{T}{t_1}\right\rfloor+1}\geq \textbf{C}_0\exp\left(-\frac{\textbf{C}_1T}{t_1^{2-\beta/\alpha}}\right).\]
Therefore we have
\[
P\left(\left\lbrace\vert u(t,\textbf{x})\vert\leq \mathcal{C}_3t_1^{\frac{\alpha-\beta}{2\alpha}}, t\in[0,T],\textbf{x}\in\mathbb{T}^d\right\rbrace\right)\geq\textbf{C}_0\exp\left(-\frac{\textbf{C}_1T}{t_1^{2-\beta/\alpha}}\right),
\]
and replacing $\mathcal{C}_3t_1^{\frac{\alpha-\beta}{2\alpha}}$ with $\varepsilon_3$ gives
\begin{equation}
\label{lower2}
P\left(\left\lbrace\vert u(t,\textbf{x})\vert\leq \varepsilon_3, t\in[0,T],\textbf{x}\in\mathbb{T}^d\right\rbrace\right)\geq\textbf{C}_0\exp\left(-\frac{\textbf{C}_1T}{\varepsilon_3^{\frac{4\alpha-2\beta}{\alpha-\beta}}}\right).
\end{equation}
We derive that, from \eqref{c0},
\[
\varepsilon_3=\mathcal{C}_3t_1^\frac{\alpha-\beta}{2\alpha}=\mathcal{C}_3\left(c_0\varepsilon^4\right)^\frac{\alpha-\beta}{2\alpha}=\mathcal{C}_3\left(\mathcal{C}_0\varepsilon^{\frac{2\alpha d}{\beta}}\right)^\frac{\alpha-\beta}{2\alpha}<\mathcal{C}_3\left(\mathcal{C}_0\varepsilon_1^{\frac{2\alpha d}{\beta}}\right)^\frac{\alpha-\beta}{2\alpha},
\]
so \eqref{lower2} provides the lower bound in Theorem \ref{thm}(a) whenever $\varepsilon_0<\mathcal{C}_3\left(\mathcal{C}_0\varepsilon_1^{\frac{2\alpha d}{\beta}}\right)^\frac{\alpha-\beta}{2\alpha}$. 

With such an $\varepsilon_0<\left(\mathcal{C}_0\varepsilon_1^{\frac{2\alpha d}{\beta}}\right)^\frac{\alpha-\beta}{2\alpha}\cdot \min\left\lbrace\mathcal{C}_3,1\right\rbrace$, we conclude Theorem \ref{thm}(a) from Proposition \ref{prop}, and with \eqref{F1}, Theorem \ref{thm}(b) follows in a similar way. The rest of this paper is devoted to the proof of Proposition \ref{prop}.
\end{proof}

\section{Preliminaries}
\label{estimates}
In this section, we provide some preliminary results that are used to prove Proposition \ref{prop}.

\subsection{Heat Kernel Estimates}
For $\textbf{x}\in \mathbb{R}^d$, the fundamental solution $p(t,\textbf{x})$ to the equation $\partial_tp(t,\textbf{x})=-(-\Delta)^{\alpha/2}p(t,\textbf{x})$ on $\mathbb{R}^d$ does not have a closed form unless $\alpha=2$, yet it is a smooth function determined by its Fourier transform $\hat{p}(t,\nu)$ in $\textbf{x}$, i.e.,
\begin{equation*}
\hat{p}(t,\nu):=\int_{\R^d}p(t,\textbf{x})\exp(-2\pi i\nu\cdot \textbf{x})d\textbf{x}=\exp\left(-t(2\pi\vert \nu\vert)^\alpha\right),\quad \nu\in\R^d.
\end{equation*}
For $\textbf{x}\in \mathbb{T}^d$, from the standard Fourier decomposition, we use $\bar{p}(t,\textbf{x})$ to denote the fundamental solution on $\mathbb{T}^d$ in \eqref{equation1} as
\begin{equation}
\label{pfourier}
\bar{p}(t,\textbf{x})=\sum_{\textbf{n}\in\mathbb{Z}^d}\exp\left(-(2\pi\vert \textbf{n}\vert)^\alpha t\right)\exp(2\pi i\textbf{n}\cdot \textbf{x}),
\end{equation}
and the summation converges in $L^2$ sense.

The following lemma gives two estimates on $\bar{p}(t,\textbf{x})$, which are similar to Lemma $2.1$ and Lemma $2.2$ in \cite{li2017holder}.

\begin{lemma}\label{pdiff} For all $t\geq s>0$ and $\textbf{x}\in \mathbb{T}^d$, there exist constants $C,C'>0$ depending entirely on $\alpha,d$ such that
\begin{equation}
\label{pdiffspace}
\int_{\mathbb{T}^d}\vert \bar{p}(t,\textbf{y}-\textbf{x})-\bar{p}(t,\textbf{y})\vert d\textbf{y}\leq \frac{C\vert \textbf{x}\vert}{t^{1/\alpha}}\wedge 2,
\end{equation}
\begin{equation}
\label{pdifftime}
\int_{\mathbb{T}^d}\vert \bar{p}(t,\textbf{x})-\bar{p}(s,\textbf{x})\vert d\textbf{x}\leq C'\log\left(\frac{t}{s}\right)\wedge 2.
\end{equation}
\end{lemma}
\begin{proof}
We begin with inequality \eqref{pdiffspace},
\begin{equation}
\begin{split}
\label{pdiffproof}
\int_{\mathbb{T}^d}\vert \bar{p}(t,{\textbf{y}-\textbf{x}})-\bar{p}(t,\textbf{y})\vert d\textbf{y}&=\int_{\mathbb{T}^d}\left\vert \sum_{\textbf{n}\in \Z^d}\left[p(t,\textbf{y}-\textbf{x}+\textbf{n})-p(t,\textbf{y}+\textbf{n})\right]\right\vert d\textbf{y}\\
&\leq \int_{\mathbb{T}^d}\sum_{\textbf{n}\in \Z^d}\left\vert p(t,\textbf{y}-\textbf{x}+\textbf{n})-p(t,\textbf{y}+\textbf{n})\right\vert d\textbf{y}\\
&=\int_{\R^d}\vert p(t,\textbf{y}-\textbf{x})-p(t,\textbf{y})\vert d\textbf{y}\\
&\leq \int_{\R^d} \vert \textbf{x}\vert\cdot\sup_{c\in[0,1]}\left\vert\nabla_{\textbf{y}} p(t,\textbf{y}-c\textbf{x})\right\vert d\textbf{y},
\end{split}
\end{equation}
where we use the Mean Value Theorem for the last inequality.

Using the inequality from Lemma 3 and Lemma 5 in \cite{bogdan2007estimates} along with equation $(2.3)$ from \cite{jakubowski2016stable}, we obtain the following bound:
\begin{equation}
\label{pgradiant}
\left\vert\nabla_{\textbf{z}} p(t,\textbf{z})\right\vert \leq C(d,\alpha)\vert \textbf{z}\vert\left(\frac{t}{\vert \textbf{z}\vert^{d+2+\alpha}}\wedge t^{-(d+2)/\alpha}\right)\leq C(d,\alpha)\frac{t\vert \textbf{z}\vert}{(t^{1/\alpha}+\vert \textbf{z}\vert)^{d+2+\alpha}}.
\end{equation}
Plugging \eqref{pgradiant} in \eqref{pdiffproof} yields
\begin{equation*}
\begin{split}
\int_{\mathbb{T}^d}\vert \bar{p}(t,{\textbf{y}-\textbf{x}})-\bar{p}(t,\textbf{y})\vert d\textbf{y}&\leq C(d,\alpha)\vert \textbf{x}\vert\int_{\R^d}\frac{t\vert \textbf{y}\vert}{(t^{1/\alpha}+\vert \textbf{y}\vert)^{d+2+\alpha}}d\textbf{y}\\
&= C(d,\alpha)\vert \textbf{x}\vert\int_0^{\infty}\frac{tx}{(t^{1/\alpha}+x)^{d+2+\alpha}}x^{d-1}dx\\
&=\frac{C(d,\alpha)\vert \textbf{x}\vert}{t^{1/\alpha}}\int_0^{\infty}\frac{w^d}{(1+w)^{d+2+\alpha}}dw\\
&=\frac{C(d,\alpha)\vert \textbf{x}\vert}{t^{1/\alpha}}.
\end{split}
\end{equation*} 
Clearly, $\int_{\mathbb{T}^d}\vert \bar{p}(t,{\textbf{y}-\textbf{x}})-\bar{p}(t,\textbf{y})\vert d\textbf{y}\leq 2$, so that \eqref{pdiffspace} follows. 

For inequality \eqref{pdifftime}, we have
\begin{equation}
\label{pdiffproof2}
\begin{split}
\int_{\mathbb{T}^d}\vert \bar{p}(t,{\textbf{x}})-\bar{p}(s,\textbf{x})\vert d\textbf{x}&=\int_{\mathbb{T}^d}\left\vert \sum_{\textbf{n}\in \Z^d}\left[p(t,\textbf{x}+\textbf{n})-p(s,\textbf{x}+\textbf{n})\right]\right\vert d\textbf{x}\\
&\leq \int_{\mathbb{T}^d}\sum_{\textbf{n}\in \Z^d}\left\vert p(t,\textbf{x}+\textbf{n})-p(s,\textbf{x}+\textbf{n})\right\vert d\textbf{x}\\
&=\int_{\R^d}\vert p(t,\textbf{x})-p(s,\textbf{x})\vert d\textbf{x}\\
&=\int_{\R^d}\int_s^t\vert\partial_r p(r,\textbf{x})\vert drd\textbf{x}.
\end{split}
\end{equation}
Proposition $2.1$ in \cite{vazquez2017classical} provides an estimate for the kernel $A:=(-\Delta)^{\sigma/2}p$ in terms of the $\sigma$-parabolic "distance" $\vert Y\vert_\sigma:=\left(\vert\textbf{x}\vert^2 + \vert t\vert^{2/\sigma}\right)^{1/2}$. Specifically, it establishes that
\begin{equation}
\label{ptimegradient}
\vert(-\Delta)^{\alpha/2} p(r,\textbf{x})\vert = \vert A(r,\textbf{x})\vert\leq \frac{C(d,\alpha)}{(r^{2/\alpha}+\vert \textbf{x}\vert^2)^{\frac{d+\alpha}{2}}},
\end{equation}
where $C(d,\alpha)$ is a constant depending on the dimension $d$ and and parameter $\alpha$. Hence, applying \eqref{equation1}, \eqref{ptimegradient} and Fubini's theorem to \eqref{pdiffproof2} yields
\begin{align*}
\int_{\mathbb{T}^d}\vert \bar{p}(t,{\textbf{x}})-\bar{p}(s,\textbf{x})\vert d\textbf{x}&\leq \int_{\R^d}\int_s^t\vert\partial_r p(r,\textbf{x})\vert drd\textbf{x}\\
&=\int_{\R^d}\int_s^t\vert(-\Delta)^{\alpha/2} p(r,\textbf{x})\vert drd\textbf{x}\\
&\leq C(d,\alpha)\int_{\R^d}\int_s^t\frac{1}{(r^{2/\alpha}+\vert \textbf{x}\vert^2)^{\frac{d+\alpha}{2}}} drd\textbf{x}\\
&=C(d,\alpha)\int_s^t\int_{\R^d} \frac{1}{(r^{2/\alpha}+\vert \textbf{x}\vert^2)^{\frac{d+\alpha}{2}}} d\textbf{x}dr\\
&=C(d,\alpha)\int_s^t\int_0^\infty \frac{x^{d-1}}{(r^{2/\alpha}+x^2)^{\frac{d+\alpha}{2}}} dxdr\\
&=C(d,\alpha)\int_s^t\frac{dr}{r}\int_0^\infty \frac{w^{d-1}}{(1+w^2)^{\frac{d+\alpha}{2}}} dw\\
&=C(d,\alpha)\left(\log(t)-\log(s)\right).
\end{align*} 
Similarly, $\int_{\mathbb{T}^d}\vert \bar{p}(t,{\textbf{x}})-\bar{p}(s,\textbf{x})\vert d\textbf{x}\leq 2$, so that \eqref{pdifftime} follows.
\end{proof}

\subsection{Noise Term Estimates}
We denote the second integral of \eqref{mild}, i.e., the noise term, by
\begin{equation}
\label{noiseterm}
N(t,\textbf{x}):=\int_{[0,t]\times\mathbb{T}^d} \bar{p}(t-s,\textbf{x}-\textbf{y})\sigma(s,\textbf{y},u(s,\textbf{y}))F(dsd\textbf{y})
\end{equation}
whose regularities will be estimated in the following two lemmas.

\begin{lemma}\label{spatialregularity} There exists a constant $C(d,\alpha,\beta)>0$ independent of $\mathcal{C}_2$ in \eqref{hypothesis2} such that for any $\xi\in\left(0,\frac{1}{\alpha}\wedge\frac{\alpha-\beta}{\alpha}\right)$, $t\in[0,1]$ and $\textbf{x},\textbf{y}\in \mathbb{T}^d$, we have
\begin{equation*}
\mathbb{E}\left[(N(t,\textbf{x})-N(t,\textbf{y}))^2\right]\leq C(d,\alpha,\beta)\mathcal{C}_2^2\vert \textbf{x}-\textbf{y}\vert^{\alpha\xi}.
\end{equation*}
\end{lemma}
\begin{proof}
To simplify our computations, we fix $t,\textbf{x}$ and $\textbf{y}$, and define, for $0<s<t$,
\[K_s(\textbf{z}):=\bar{p}(t-s,\textbf{x}-\textbf{z})-\bar{p}(t-s,\textbf{y}-\textbf{z}).\]

By a similar argument as on page 2871 of \cite{bezdek2016weak}, we use Burkholder's inequality, the Cauchy-Schwarz inequality and take the absolute value inside the integral and get, for the case $k=2$,
\begin{equation}
\label{spatialregularityproof}
\begin{split}
&\mathbb{E}\left[(N(t,\textbf{x})-N(t,\textbf{y}))^2\right]\\
&=\mathbb{E}\left[\left(\int_{[0,t]\times\mathbb{T}^d}K_s(\textbf{z})\sigma(s,\textbf{z},u(s,\textbf{z}))F(dsd\textbf{z})\right)^2\right]\\
&\leq C\mathbb{E}\left\vert\int_0^t\int_{\mathbb{T}^d}\int_{\mathbb{T}^d}K_s(\textbf{z})K_s(\textbf{w})\Lambda(\textbf{w}-\textbf{z})\sigma(s,\textbf{z},u(s,\textbf{z}))\sigma(s,\textbf{w},u(s,\textbf{w}))d\textbf{w}d\textbf{z}ds\right\vert\\
&\leq C\sup_{r,\textbf{u}}\mathbb{E}\left[\sigma(r,\textbf{u},u(r,\textbf{u}))^2\right]\int_0^t\int_{\mathbb{T}^d}\int_{\mathbb{T}^d}\vert K_s(\textbf{z})\vert\cdot \vert K_s(\textbf{w})\vert\cdot\Lambda(\textbf{w}-\textbf{z})d\textbf{w}d\textbf{z}ds\\
&\leq C\mathcal{C}_2^2\int_0^t\int_{\mathbb{T}^d}\int_{\mathbb{T}^d}\vert K_s(\textbf{z})\vert\cdot[\bar{p}(t-s,\textbf{x}-\textbf{w})+\bar{p}(t-s,\textbf{y}-\textbf{w})]\Lambda(\textbf{w-z})d\textbf{w}d\textbf{z}ds.
\end{split}
\end{equation}
Then we apply the standard Fourier decomposition \eqref{pfourier} and \eqref{lambdabound} to estimate the spatial convolution,
\begin{equation}
\label{pconv}
\begin{split}
\int_{\mathbb{T}^d}\bar{p}(t-s,\textbf{x}-\textbf{w})\Lambda(\textbf{w}-\textbf{z})d\textbf{w}&=\sum_{\textbf{n}\in\Z^d}\lambda(\textbf{n})\exp(-(2\pi\vert \textbf{n}\vert)^\alpha(t-s))\exp(2\pi i\textbf{n}\cdot(\textbf{x}-\textbf{z}))\\
&\leq \sum_{\textbf{n}\in\Z^d}\lambda(\textbf{n})\exp(-(2\pi\vert \textbf{n}\vert)^\alpha(t-s))\\
&\leq C(d,\beta)\sum_{\textbf{n}\in\Z^d}\vert \textbf{n}\vert^{-d+\beta}\exp(-(2\pi\vert \textbf{n}\vert)^\alpha(t-s))\\
&\leq C(d,\beta)\int_{\R^d} \vert\textbf{x}\vert^{-d+\beta}\exp(-(2\pi \vert \textbf{x}\vert)^\alpha(t-s))d\textbf{x}\\
&\leq C(d,\beta)\int_0^\infty x^{-d+\beta}\exp(-(2\pi x)^\alpha(t-s))x^{d-1}dx\\
&=C(d,\alpha,\beta)(t-s)^{-\beta/\alpha}\int_0^\infty w^{\frac{\beta}{\alpha}-1}\exp(-w)dw\\
&=C(d,\alpha,\beta)(t-s)^{-\beta/\alpha},
\end{split}
\end{equation}
where $\int_0^\infty w^{\frac{\beta}{\alpha}-1}\exp(-w)dw$ is the Gamma function valued at $\frac{\beta}{\alpha}$. Similarly, we get a bound for $\int_{\mathbb{T}^d}\bar{p}(t-s,\textbf{y}-\textbf{w})\Lambda(\textbf{w}-\textbf{z})d\textbf{w}$ as well. 

Plugging \eqref{pconv} and Lemma \ref{pdiff} in \eqref{spatialregularityproof}, and since $2\wedge Cx\leq \max\{2,C\}x^{\alpha\xi}$ for all $x>0,\xi\in(0,1/\alpha)$, we derive
\begin{equation*}
\begin{split}
\mathbb{E}\left[(N(t,\textbf{x})-N(t,\textbf{y}))^2\right]&\leq C(d, \alpha, \beta)\mathcal{C}_2^2\int_0^t\int_{\mathbb{T}^d}\vert K_s(\textbf{z})\vert(t-s)^{-\beta/\alpha}d\textbf{z}ds\\
&\leq C(d, \alpha, \beta)\mathcal{C}_2^2\int_0^t(t-s)^{-\beta/\alpha}\left(\frac{C\vert \textbf{x}-\textbf{y}\vert}{(t-s)^{1/\alpha}}\wedge 2\right)ds\\
&\leq C(d, \alpha, \beta)\mathcal{C}_2^2\vert \textbf{x}-\textbf{y}\vert^{\alpha\xi}\int_0^t(t-s)^{-\xi-\beta/\alpha}ds\\
&\leq  C(d, \alpha, \beta)\mathcal{C}_2^2\vert \textbf{x}-\textbf{y}\vert^{\alpha\xi}.
\end{split}
\end{equation*}
Note that the integral $\int_0^t(t-s)^{-\xi-\beta/\alpha}ds$ converges when $\xi<\frac{\alpha-\beta}{\alpha}$.
\end{proof}

\begin{lemma}\label{timeregularity} There exists a constant $C(d, \alpha, \beta)>0$ independent of $\mathcal{C}_2$ in \eqref{hypothesis2} such that for any $\zeta\in\left(0,\frac{\alpha-\beta}{\alpha}\right)$, $1\geq t\geq s>0$ and $\textbf{x},\textbf{y}\in \mathbb{T}^d$, we have
\begin{equation*}
\mathbb{E}\left[(N(t,\textbf{x})-N(s,\textbf{x}))^2\right]\leq C(d, \alpha, \beta)\mathcal{C}_2^2\vert t-s\vert^{\zeta}.
\end{equation*}
\end{lemma}
\begin{proof} As in Lemma \ref{spatialregularity}, we fix $s$, $t$ and $\textbf{x}$, and define, for $0<r<s<t$,
\begin{equation*}
L_r(\textbf{z})=\bar{p}(t-r,\textbf{x}-\textbf{z})-\bar{p}(s-r,\textbf{x}-\textbf{z}).
\end{equation*}
Following the fact that the Gaussian noise $\dot{F}(t,\textbf{x})$ is independent on disjoint time intervals, we obtain
\begin{equation}
\label{timeregularityproof}
\begin{split}
\mathbb{E}&[(N(t,\textbf{x})-N(s,\textbf{x}))^2]\\
&=\mathbb{E}\left[\left(\int_0^s\int_{\mathbb{T}^d}L_r(\textbf{z})\sigma(r,\textbf{z},u(r,\textbf{z}))F(d\textbf{z}dr)\right.\right.\\
&\hspace{5cm}\left.\left.+\int_s^t\int_{\mathbb{T}^d}\bar{p}(t-r,\textbf{x}-\textbf{z})\sigma(r,\textbf{z},u(r,\textbf{z}))F(d\textbf{z}dr)\right)^2\right]\\
&=\mathbb{E}\left[\left(\int_0^s\int_{\mathbb{T}^d}L_r(\textbf{z})\sigma(r,\textbf{z},u(r,\textbf{z}))F(d\textbf{z}dr)\right)^2\right]\\
&\hspace{4cm}+\mathbb{E}\left[\left(\int_s^t\int_{\mathbb{T}^d}\bar{p}(t-r,\textbf{x}-\textbf{z})\sigma(r,\textbf{z},u(r,\textbf{z}))F(d\textbf{z}dr)\right)^2\right]\\
&=:I_1+I_2.
\end{split}
\end{equation}

We study $I_1$ and $I_2$ separately. For $I_1$, as in \eqref{spatialregularityproof}, we get
\begin{equation}
\label{I1}
\begin{split}
I_1&\leq C\sup_{r,\textbf{u}}\mathbb{E}\left[\sigma(r,\textbf{u},u(r,\textbf{u}))^2\right]\int_0^s\int_{\mathbb{T}^d}\int_{\mathbb{T}^d}\vert L_r(\textbf{z})\vert\cdot \vert L_r(\textbf{w})\vert\cdot\Lambda(\textbf{w}-\textbf{z})d\textbf{w}d\textbf{z}dr\\
&\leq C\mathcal{C}_2^2\int_0^s\int_{\mathbb{T}^d}\int_{\mathbb{T}^d}\vert L_r(\textbf{z})\vert\cdot[\bar{p}(t-r,\textbf{x}-\textbf{w})+\bar{p}(s-r,\textbf{x}-\textbf{w})]\Lambda(\textbf{w-z})d\textbf{w}d\textbf{z}dr.
\end{split}
\end{equation}
Applying \eqref{pconv} and Lemma \ref{pdiff} to \eqref{I1}, and since $2\wedge C'\log(1+x)\leq \max\{2, C'\}x^{\zeta}$ for all $x>0,\zeta\in(0,1)$, we derive
\begin{equation}
\begin{split}
\label{timeregularityproof2}
I_1&\leq C(d,\alpha,\beta)\mathcal{C}_2^2\int_0^s\left(C'\log\left(\frac{t-r}{s-r}\right)\wedge 2\right)\cdot\left[(t-r)^{-\beta/\alpha}+(s-r)^{-\beta/\alpha}\right]dr\\
&\leq C(d,\alpha,\beta)\mathcal{C}_2^2\int_0^s\left(C'\log\left(\frac{t-r}{s-r}\right)\wedge 2\right)\cdot\left[(s-r)^{-\beta/\alpha}+(s-r)^{-\beta/\alpha}\right]dr\\
&\leq C(d,\alpha,\beta)\mathcal{C}_2^2\int_0^s\left(C'\log\left(\frac{t-s+u}{u}\right)\wedge 2\right)u^{-\beta/\alpha}du\\
&\leq C(d,\alpha,\beta)\mathcal{C}_2^2(t-s)^\zeta\int_0^su^{-\beta/\alpha-\zeta}du\\
&\leq C(d,\alpha,\beta)\mathcal{C}_2^2(t-s)^\zeta.
\end{split}
\end{equation}
Note that the integral $\int_0^su^{-\beta/\alpha-\zeta}du$ converges when $\zeta<\frac{\alpha-\beta}{\alpha}$. 

For $I_2$, as in \eqref{spatialregularityproof}, we get
\begin{equation}
\label{I2}
\begin{split}
I_2&\leq C\sup_{r,\textbf{u}}\mathbb{E}\left[\sigma(r,\textbf{u},u(r,\textbf{u}))^2\right]\\
&\hspace{3cm}\cdot\int_s^t\int_{\mathbb{T}^d}\int_{\mathbb{T}^d}\bar{p}(t-r,\textbf{x}-\textbf{z})\bar{p}(t-r,\textbf{x}-\textbf{w})\Lambda(\textbf{w}-\textbf{z})d\textbf{w}d\textbf{z}dr\\
&\leq C\mathcal{C}_2^2\int_s^t\int_{\mathbb{T}^d}\int_{\mathbb{T}^d}\bar{p}(t-r,\textbf{x}-\textbf{z})\bar{p}(t-r,\textbf{x}-\textbf{w})\Lambda(\textbf{w}-\textbf{z})d\textbf{w}d\textbf{z}dr.
\end{split}
\end{equation}

As in \eqref{pconv}, we use the standard Fourier decomposition \eqref{pfourier} and \eqref{lambdabound} to estimate the spatial convolution in \eqref{I2}, 
\begin{equation}
\label{ppconv}
\begin{split}
\int_{\mathbb{T}^d}\int_{\mathbb{T}^d}\bar{p}(t-r,\textbf{x}-\textbf{z})&\bar{p}(t-r,\textbf{x}-\textbf{w})\Lambda(\textbf{w}-\textbf{z})d\textbf{w}d\textbf{z}\\
&=\sum_{\textbf{n}\in\Z^d}\lambda(\textbf{n})\exp\left(-2(2\pi \vert \textbf{n}\vert)^\alpha(t-r)\right)\\
&\leq C(d,\beta)\sum_{\textbf{n}\in\Z^d}\vert \textbf{n}\vert^{-d+\beta}\exp\left(-2(2\pi\vert \textbf{n}\vert)^\alpha(t-r)\right)\\
&\leq C(d,\beta)\int_{\R^d} \vert\textbf{x}\vert^{-d+\beta}\exp(-2(2\pi \vert \textbf{x}\vert)^\alpha(t-r))d\textbf{x}\\
&\leq C(d,\beta)\int_0^\infty x^{-d+\beta}\exp(-2(2\pi x)^\alpha(t-r))x^{d-1}dx\\
&=C(d,\alpha,\beta)(t-r)^{-\beta/\alpha}\int_0^\infty w^{\frac{\beta}{\alpha}-1}\exp(-w)dw\\
&=C(d,\alpha,\beta)(t-r)^{-\beta/\alpha},
\end{split}
\end{equation}
where $\int_0^\infty w^{\frac{\beta}{\alpha}-1}\exp(-w)dw$ is the Gamma function valued at $\frac{\beta}{\alpha}$. Plugging it in \eqref{I2} yields
\begin{equation}
\label{I2equ}
\begin{split}
I_2&\leq C(d,\alpha,\beta)\mathcal{C}_2^2\int_s^t(t-r)^{-\beta/\alpha}dr\\
&=C(d,\alpha,\beta)\mathcal{C}_2^2(t-s)^{\frac{\alpha-\beta}{\alpha}}\\
&<C(d,\alpha,\beta)\mathcal{C}_2^2(t-s)^{\zeta}.
\end{split}
\end{equation}

By \eqref{timeregularityproof}, \eqref{timeregularityproof2} and \eqref{I2equ}, we therefore conclude
\begin{equation*}
\label{timereg}
\mathbb{E}\left[(N(t,\textbf{x})-N(s,\textbf{x}))^2\right]\leq C(d,\alpha,\beta)\mathcal{C}_2^2(t-s)^{\zeta}.
\end{equation*}
\end{proof}

\begin{lemma}There exist constants $C_1,C_2,C_3,C_4>0$ depending only on $\alpha,\beta,d$ and $\mathcal{C}_2$ in \eqref{hypothesis2} such that for all $0\leq s<t\leq 1$, $\textbf{x},\textbf{y}\in\mathbb{T}^d$, $\xi\in\left(0,\frac{1}{\alpha}\wedge\frac{\alpha-\beta}{\alpha}\right),\zeta\in\left(0,\frac{\alpha-\beta}{\alpha}\right)$, and $\kappa>0$, we have
\begin{equation}
\label{spacep}
P(\vert N(t,\textbf{x})-N(t,\textbf{y})\vert>\kappa)\leq C_1\exp\left(-\frac{C_2\kappa^2}{\mathcal{C}_2^2\vert \textbf{x}-\textbf{y}\vert^{\alpha\xi}}\right),
\end{equation}
\begin{equation}
\label{timep}
P(\vert N(t,\textbf{x})-N(s,\textbf{x})\vert>\kappa)\leq C_3\exp\left(-\frac{C_4\kappa^2}{\mathcal{C}_2^2\vert t-s\vert^{\zeta}}\right).
\end{equation}
\end{lemma}
\begin{proof}
For a fixed $t$, we define
\[N_t(s,\textbf{x}):=\int_{[0,s]\times\mathbb{T}^d} \bar{p}(t-r,\textbf{x}-\textbf{y})\sigma(r,\textbf{y},u(r,\textbf{y}))F(drd\textbf{y}).\]
Note that $N_t(t,\textbf{x})=N(t,\textbf{x})$ and $N_t(s,\textbf{x})$ is a continuous $\mathcal{F}_s^F$ adapted martingale in $s\leq t$ since the integrand does not depend on $s$. For fixed $t,\textbf{x}$ and $\textbf{y}$, let
\[M_s:=N_t(s,\textbf{x})-N_t(s,\textbf{y})=\int_{[0,s]\times\mathbb{T}^d} [\bar{p}(t-r,\textbf{x}-\textbf{z})-\bar{p}(t-r,\textbf{y}-\textbf{z})]\sigma(r,\textbf{z},u(r,\textbf{z}))F(drd\textbf{z}),\]
and it is easy to check that $M_t=N(t,\textbf{x})-N(t,\textbf{y})$. As $M_s$ is a continuous $\mathcal{F}_s^F$ adapted martingale with $M_0=0$, it is a time changed Brownian motion, i.e.,
\[M_t=B_{\langle M\rangle_t},\]
and Lemma \ref{spatialregularity} gives a uniform bound as
\[\langle M\rangle_t\leq C\mathcal{C}_2^2\vert \textbf{x}-\textbf{y}\vert^{\alpha\xi}.\]
Therefore, by the reflection principle for Brownian motion $B_{\langle M\rangle_t}$, we obtain
\begin{align*}
P(N(t,\textbf{x})-N(t,\textbf{y})>\kappa)&=P(M_t>\kappa)=P(B_{\langle M\rangle_t}>\kappa)\\
&\leq P\left(\sup_{s\leq C\mathcal{C}_2^2\vert \textbf{x}-\textbf{y}\vert^{\alpha\xi}}B_s>\kappa\right)=2P\left(B_{C\mathcal{C}_2^2\vert \textbf{x}-\textbf{y}\vert^{\alpha\xi}}>\kappa\right)\\
&\leq C_1\exp\left(-\frac{C_2\kappa^2}{\mathcal{C}_2^2\vert \textbf{x}-\textbf{y}\vert^{\alpha\xi}}\right).
\end{align*}
Switching $\textbf{x}$ and $\textbf{y}$ yields
\begin{align*}
P(-N(t,\textbf{x})+N(t,\textbf{y})>\kappa)&=P(M_t<-\kappa)\leq 2P\left(B_{C\mathcal{C}_2^2\vert \textbf{x}-\textbf{y}\vert^{\alpha\xi}}<-\kappa\right)\\
&\leq C_1\exp\left(-\frac{C_2\kappa^2}{\mathcal{C}_2^2\vert \textbf{x}-\textbf{y}\vert^{\alpha\xi}}\right).
\end{align*}
Consequently, for $\forall\xi\in\left(0,\frac{1}{\alpha}\wedge\frac{\alpha-\beta}{\alpha}\right)$, 
\[
P(\vert N(t,\textbf{x})-N(t,\textbf{y})\vert>\kappa)\leq C_1\exp\left(-\frac{C_2\kappa^2}{\mathcal{C}_2^2\vert \textbf{x}-\textbf{y}\vert^{\alpha\xi}}\right),
\]
which completes the proof of \eqref{spacep}. 

For a fixed $\textbf{x}$, we define 
\[U_{q_1}:=\int_{[0,{q_1}]\times\mathbb{T}^d} [\bar{p}(t-r,\textbf{x}-\textbf{y})-\bar{p}(s-r,\textbf{x}-\textbf{y})]\sigma(r,\textbf{y},u(r,\textbf{y}))F(drd\textbf{y})\]
where $0\leq q_1\leq s$. Note $U_{q_1}$ is a continuous $\mathcal{F}_{q_1}^F$ adapted martingale with $U_0=0$. In addition, we define
\[V_{q_2}:=\int_{[0,{q_2}]\times\mathbb{T}^d}\bar{p}(t-s-r,\textbf{x}-\textbf{y})\sigma(r+s,\textbf{y},u(r+s,\textbf{y}))F(drd\textbf{y})\]
where $0\leq q_2\leq t-s$. Note that $V_{q_2}$ is a continuous $\mathcal{F}_{q_2}^F$ adapted martingale with $V_0=0$. Thus, both $U_{q_1}$ and $V_{q_2}$ are time changed Brownian motions, i.e.,
\[U_s=B_{\langle U\rangle_s}\text{~and~}V_{t-s}=B'_{\langle V\rangle_{t-s}},\]
where $B$, $B'$ are two different Brownian motions. It is clear that $N(t,\textbf{x})-N(s,\textbf{x})=U_s+V_{t-s}$, which leads to
\[P(N(t,\textbf{x})-N(s,\textbf{x})>\kappa)\leq P(U_s>\kappa/2)+P(V_{t-s}>\kappa/2),\]
and Lemma \ref{timeregularity} provides a uniform bound as
\[\langle U\rangle_s\leq C\mathcal{C}_2^2(t-s)^{\zeta}\text{~and~}\langle V\rangle_{t-s}\leq C\mathcal{C}_2^2(t-s)^{\zeta}.\]
By the reflection principle for Brownian motions $B_{\langle U\rangle_s}$ and $B'_{\langle V\rangle_{t-s}}$, we get
\begin{align*}
P(N(t,\textbf{x})-N(s,\textbf{x})>\kappa)&\leq P\left(B_{\langle U\rangle_s}>\kappa/2\right)+P\left(B'_{\langle V\rangle_{t-s}}>\kappa/2\right)\\
&\leq 2P\left(\sup_{r\leq C\mathcal{C}_2^2\vert t-s\vert^{\zeta}}B_r>\frac{\kappa}{2}\right)=4P\left(B_{C\mathcal{C}_2^2\vert t-s\vert^{\zeta}}>\frac{\kappa}{2}\right)\\
&\leq C_3\exp\left(-\frac{C_4\kappa^2}{\mathcal{C}_2^2\vert t-s\vert^{\zeta}}\right),
\end{align*}
and
\begin{align*}
P(-N(t,\textbf{x})+N(s,\textbf{x})>\kappa)&\leq P(U_s<-\kappa/2)+P(V_{t-s}<-\kappa/2)\\
&\leq 4P\left(B_{C\mathcal{C}_2^2\vert t-s\vert^{\zeta}}<-\frac{\kappa}{2}\right)\\
&\leq C_3\exp\left(-\frac{C_4\kappa^2}{\mathcal{C}_2^2\vert t-s\vert^{\zeta}}\right).
\end{align*}
Consequently, for $\forall\zeta\in\left(0,\frac{\alpha-\beta}{\alpha}\right)$, 
\[
P(\vert N(t,\textbf{x})-N(s,\textbf{x})\vert>\kappa)\leq C_3\exp\left(-\frac{C_4\kappa^2}{\mathcal{C}_2^2\vert t-s\vert^{\zeta}}\right),
\]
which completes the proof of \eqref{timep}.
\end{proof}

\begin{definition}
\label{definition}
Given a grid,
\[\mathbb{G}_n=\left\lbrace\left(\frac{j}{2^{2n}},\frac{k_1}{2^{n}},...,\frac{k_d}{2^{n}}\right): 0\leq j\leq 2^{2n},0\leq k_1,...,k_d\leq 2^n~\text{for}~j,k_1,...,k_d\in\Z\right\rbrace,\]
we write
\[\left(t_j^{(n)},x_{k_1}^{(n)},...,x_{k_d}^{(n)}\right)=\left(\frac{j}{2^{2n}},\frac{k_1}{2^n},...,\frac{k_d}{2^n}\right).\]
Two points $\left(t_j^{(n)},x_{k_1}^{(n)},...,x_{k_d}^{(n)}\right)$ and $\left(t_{j'}^{(n)},x_{k'_1}^{(n)},...,x_{k'_d}^{(n)}\right)$ are called the \textbf{nearest neighbors} if either
\begin{enumerate}
\item[\textbf{1}.] $j=j',\vert k_i-k'_i\vert=1 \text{~for only one~} i \text{~and~} k_l=k'_l \text{~for the other indices~}l, or$
\item[\textbf{2}.] $\vert j-j'\vert=1~\text{and}~k_i=k'_i \text{~for all~} i.$
\end{enumerate}
\end{definition}

The following lemma is analogous to Lemma 3.4 in \cite{athreya2021small} and Lemma 2.5 in \cite{foondun2022small}, which is crucial for estimating small ball probability.

\begin{lemma}\label{larged}There exist constants $C_5,C_6>0$ depending only on $\alpha,\beta,d$ and $\mathcal{C}_2$ in \eqref{hypothesis2} such that for all $\gamma,\kappa,\varepsilon>0$ and $\gamma\varepsilon^4\leq 1$, we have
\begin{equation*}
P\left(\sup_{\substack{0\leq t\leq\gamma\varepsilon^4\\ \textbf{x}\in \left[0,\varepsilon^2\right]^d}}\vert N(t,\textbf{x})\vert>\kappa\varepsilon^{\frac{2(\alpha-\beta)}{\alpha}}\right)\leq \frac{C_5}{1\wedge \sqrt{\gamma^d}}\exp\left(-\frac{C_6\kappa^2}{\mathcal{C}_2^2\gamma^{\frac{\alpha-\beta}{\alpha}}}\right).
\end{equation*}
\end{lemma}
\begin{proof}
We fix $\gamma\geq 1$ and consider the grid
\[\mathbb{G}_n=\left\lbrace\left(\frac{j}{2^{2n}},\frac{k_1}{2^{n}},...,\frac{k_d}{2^{n}}\right): 0\leq j\leq \gamma\varepsilon^42^{2n},0\leq k_1,...,k_d\leq \varepsilon^22^n,j,k_1,...,k_d\in\Z\right\rbrace.\]
Let
\begin{equation*}
n_0=\left\lceil \log_2\left(\gamma^{-1/2}\varepsilon^{-2}\right)\right\rceil,
\end{equation*}
and it is clear that 
\begin{equation}
\label{n0}
\log_2\left(2\gamma^{-1/2}\varepsilon^{-2}\right)>n_0\geq\log_2\left(\gamma^{-1/2}\varepsilon^{-2}\right)>n_0-1.
\end{equation}
For $n< n_0$, $\mathbb{G}_n$ contains only the origin, and for $n\geq n_0$, the grid $\mathbb{G}_n$ has at most, from \eqref{n0},
\begin{equation*}
\begin{split}
\left(\gamma\varepsilon^4 2^{2n}+1\right)\cdot\left(\varepsilon^2 2^n+1\right)^d&\leq 2^{d+1+(2+d)n}\varepsilon^{2d+4}\gamma\\
&\leq 2^{d+1+(2+d)n}\varepsilon^{2d+4}\gamma^{d/2+1}\\
&\leq 2^{2d+3}2^{(2+d)(n-n_0)}
\end{split}
\end{equation*}
many points. We will choose two parameters $0<\delta_1(\alpha,\beta)<\delta_2(\alpha,\beta)<\frac{\alpha-\beta}{\alpha}$, and let
\begin{equation*}
\delta: = \delta_2-\delta_1,
\end{equation*}
which satisfies the following constraint
\begin{equation}
\label{deltaconstraint}
2\zeta\wedge\alpha\xi=\frac{2(\alpha-\beta)}{\alpha}-2\delta,
\end{equation}
where $\xi\in\left(0,\frac{1}{\alpha}\wedge\frac{\alpha-\beta}{\alpha}\right)$ and $\zeta\in\left(0,\frac{\alpha-\beta}{\alpha}\right)$. 

We now define
\begin{equation}
\label{M}
\mathcal{M}=\frac{1-2^{-\delta_1}}{(3+d)2^{\delta n_0}},
\end{equation}
and consider the event
\[A(n,\kappa)=\left\lbrace\vert N(p)-N(q)\vert\leq \kappa \mathcal{M}\varepsilon^{\frac{2(\alpha-\beta)}{\alpha}} 2^{-\delta_1n}2^{\delta_2n_0}\text{ for all nearest neighbors $p,q\in \mathbb{G}_n$}\right\rbrace.\]
If $p,q\in \mathbb{G}_n$ are nearest neighbors in the sense of Definition \ref{definition} case $\textbf{1}$, then \eqref{spacep} implies
\[P\left(\vert N(p)-N(q)\vert> \kappa \mathcal{M}\varepsilon^{\frac{2(\alpha-\beta)}{\alpha}} 2^{-\delta_1n}2^{\delta_2n_0}\right)\leq C_1\exp\left(-\frac{C_2\kappa^2\mathcal{M}^2\varepsilon^{\frac{4(\alpha-\beta)}{\alpha}}}{2^{-n\alpha\xi}\mathcal{C}_2^2}2^{-2\delta_1n}2^{2\delta_2n_0}\right).\]
If $p,q\in \mathbb{G}_n$ are nearest neighbors in the sense of Definition \ref{definition} case $\textbf{2}$, then \eqref{timep} implies
\[P\left(\vert N(p)-N(q)\vert> \kappa \mathcal{M}\varepsilon^{\frac{2(\alpha-\beta)}{\alpha}} 2^{-\delta_1n}2^{\delta_2n_0}\right)\leq C_3\exp\left(-\frac{C_4\kappa^2\mathcal{M}^2\varepsilon^{\frac{4(\alpha-\beta)}{\alpha}}}{2^{-2n\zeta}\mathcal{C}_2^2}2^{-2\delta_1n}2^{2\delta_2n_0}\right).\]
Therefore, a union bound gives
\begin{align*}
&P(A^c(n,\kappa))\leq \sum_{\substack{p,q\in \mathbb{G}_n\\ \text{nearest neighbors}}}P\left(\vert N(p)-N(q)\vert> \kappa \mathcal{M}\varepsilon^{\frac{2(\alpha-\beta)}{\alpha}} 2^{-\delta_1n}2^{\delta_2n_0}\right)\\
&\leq C2^{(2+d)(n-n_0)}\exp\left(-\frac{C'\kappa^2\mathcal{M}^2\varepsilon^{\frac{4(\alpha-\beta)}{\alpha}}}{\mathcal{C}_2^2}2^{n(2\zeta\wedge \alpha\xi)}2^{-2\delta_1n}2^{2\delta_2n_0}\right)\\
&=C2^{(2+d)(n-n_0)}\exp\left(-\frac{C'\kappa^2\mathcal{M}^2}{\mathcal{C}_2^2}\left(\varepsilon^{\frac{4(\alpha-\beta)}{\alpha}} 2^{\frac{2n_0(\alpha-\beta)}{\alpha}}\right)2^{n(2\zeta\wedge \alpha\xi)}2^{-2\delta_1n}2^{-\frac{2n_0(\alpha-\beta)}{\alpha}}2^{2\delta_2n_0}\right)\\
&\leq C2^{(2+d)(n-n_0)}\exp\left(-\frac{C'\kappa^2\mathcal{M}^2}{\mathcal{C}_2^2\gamma^{\frac{(\alpha-\beta)}{\alpha}}}2^{(2\zeta\wedge \alpha\xi-2\delta_1)(n-n_0)}\right),
\end{align*}
where $C,C'$ are positive constants depending only on $\alpha,\beta,T,d$. The last inequality follows from that $\varepsilon^{\frac{4(\alpha-\beta)}{\alpha}} 2^\frac{2n_0(\alpha-\beta)}{\alpha}\geq\gamma^{-\frac{(\alpha-\beta)}{\alpha}}$ by the definition of $n_0$ in \eqref{n0}, and our choices of $\delta_1,\delta_2$ and $\delta$ in \eqref{deltaconstraint}. Let $A(\kappa)=\bigcap\limits_{n\geq n_0}A(n,\kappa)$ and we can bound $P(A^c(\kappa))$ by summing $P(A^c(n,\kappa))$ for all $n\geq n_0$,
\begin{align*}
P\left(A^c(\kappa)\right)&\leq\sum_{n\geq n_0}P\left(A^c(n,\kappa)\right)\\
&\leq \sum_{n\geq n_0}C2^{(2+d)(n-n_0)}\exp\left(-\frac{C'\kappa^2\mathcal{M}^2}{\mathcal{C}_2^2\gamma^{\frac{(\alpha-\beta)}{\alpha}}}2^{(2\zeta\wedge \alpha\xi-2\delta_1)(n-n_0)}\right)\\
&\leq C_5\exp\left(-\frac{C'\kappa^2\mathcal{M}^2}{\mathcal{C}_2^2\gamma^{\frac{(\alpha-\beta)}{\alpha}}}\right).
\end{align*}
From definitions \eqref{n0} and \eqref{M}, we have
\begin{align*}
P\left(A^c(\kappa)\right)&\leq C_5\exp\left(-\frac{C'\kappa^22^{-2\delta n_0}}{\mathcal{C}_2^2\gamma^{\frac{(\alpha-\beta)}{\alpha}}}\right)\\
&\leq C_5\exp\left(-\frac{C_6\kappa^2(\gamma\varepsilon^4)^\delta}{\mathcal{C}_2^2\gamma^{\frac{(\alpha-\beta)}{\alpha}}}\right)
\end{align*}
for any $\delta\in\left(0,\frac{\alpha-\beta}{\alpha}\right)$, we thus obtain
\begin{align*}
P\left(A^c(\kappa)\right)&\leq \inf_{\delta\in\left(0,\frac{\alpha-\beta}{\alpha}\right)}C_5\exp\left(-\frac{C_6\kappa^2(\gamma\varepsilon^4)^\delta}{\mathcal{C}_2^2\gamma^{\frac{(\alpha-\beta)}{\alpha}}}\right)\\
&\leq C_5\exp\left(-\frac{C_6\kappa^2}{\mathcal{C}_2^2\gamma^{\frac{\alpha-\beta}{\alpha}}}\right).
\end{align*}

Now we consider a point $(t,\textbf{x})$, which is in a grid $\mathbb{G}_n$ for some $n \geq n_0$. From arguments similar to page 128 of \cite{dalang2009minicourse}, we can find a sequence of points from the origin to $(t,\textbf{x})$ as $(0,\textbf{0})= p_0,p_1,...,p_k = (t,\textbf{x})$ such that each pair $p_jp_{j+1}$ is the nearest neighbor in some grid $\mathbb{G}_m, n_0\leq m \leq n$. At most $(3+d)$ such pairs are nearest neighbors in any given grid due to the fact that, along that path and for a given grid, there are at most $3$ nearest neighbor pairs in the time direction and maximum of $d$ nearest neighbor pairs in the space direction. On the event $A(\kappa)$, we have
\begin{align*}
\vert N(t,\textbf{x})\vert\leq \sum_{j=0}^{k-1}\vert N(p_j)-N(p_{j+1})\vert\leq (3+d)\sum_{n\geq n_0}\kappa \mathcal{M}\varepsilon^{\frac{2(\alpha-\beta)}{\alpha}} 2^{-\delta_1n}2^{\delta_2n_0}\leq\kappa\varepsilon^{\frac{2(\alpha-\beta)}{\alpha}} .
\end{align*}

Points in $\bigcup\limits_n\mathbb{G}_n$ are dense in $[0,\gamma\varepsilon^4]\times[0,\varepsilon^2]$, and we may extend $N(t,\textbf{x})$ to a continuous version using Theorem $13$ in \cite{dalang1999extending}. Therefore, for $\gamma\geq 1$,
\[
P\left(\sup_{\substack{0\leq t\leq\gamma\varepsilon^4\\ \textbf{x}\in\left[0,\varepsilon^2\right]^d}}\vert N(t,\textbf{x})\vert>\kappa\varepsilon^{\frac{2(\alpha-\beta)}{\alpha}}\right)\leq P(A^c(\kappa))\leq  C_5\exp\left(-\frac{C_6\kappa^2}{\mathcal{C}_2^2\gamma^{\frac{\alpha-\beta}{\alpha}}}\right).
\]
For $0<\gamma<1$, we divide the interval $[0,\varepsilon^2]$ into $\frac{1}{\sqrt{\gamma}}$ pieces each of length $\sqrt{\gamma}\varepsilon^2$ in each coordinate, and then a union bound implies that
\[P\left(\sup_{\substack{0\leq t\leq\gamma\varepsilon^4\\ \textbf{x}\in\left[0,\varepsilon^2\right]^d}}\vert N(t,\textbf{x})\vert>\kappa\varepsilon^{\frac{2(\alpha-\beta)}{\alpha}}\right)\leq \frac{1}{\sqrt{\gamma^d}} P\left(\sup_{\substack{0\leq t\leq\gamma\varepsilon^4\\ \textbf{x}\in\left[0,\sqrt{\gamma}\varepsilon^2\right]^d}}\vert N(t,\textbf{x})\vert>\kappa\varepsilon^{\frac{2(\alpha-\beta)}{\alpha}}\right)\]
\[=\frac{1}{\sqrt{\gamma^d}}P\left(\sup_{\substack{0\leq t\leq(\sqrt{\gamma}\varepsilon^2)^2\\ \textbf{x}\in\left[0,\sqrt{\gamma}\varepsilon^2\right]^d}}\vert N(t,\textbf{x})\vert>\frac{\kappa}{\gamma^{\frac{\alpha-\beta}{2\alpha}}}\left(\gamma^{1/4}\varepsilon\right)^{\frac{2(\alpha-\beta)}{\alpha}}\right)\leq \frac{C_5}{\sqrt{\gamma^d}}\exp\left(-\frac{C_6\kappa^2}{ \mathcal{C}_2^2\gamma^{\frac{\alpha-\beta}{\alpha}}}\right).\]
As a result,
\[
P\left(\sup_{\substack{0\leq t\leq\gamma\varepsilon^4\\ \textbf{x}\in \left[0,\varepsilon^2\right]^d}}\vert N(t,\textbf{x})\vert>\kappa\varepsilon^{\frac{2(\alpha-\beta)}{\alpha}}\right)\leq \frac{C_5}{1\wedge \sqrt{\gamma^d}}\exp\left(-\frac{C_6\kappa^2}{\mathcal{C}_2^2\gamma^{\frac{\alpha-\beta}{\alpha}}}\right).
\]
\end{proof}

\begin{remark}
\label{largeremark}
If we assume that $\sigma$ in \eqref{noiseterm} satisfies $\vert \sigma(s,\textbf{y},u(s,\textbf{y}))\vert\leq C\left(\gamma\varepsilon^4\right)^{\frac{\alpha-\beta}{2\alpha}}$, then the probability in Lemma \ref{larged} is bounded above by
\[\frac{C_5}{1\wedge \sqrt{\gamma^d}}\exp\left(-\frac{C_6\kappa^2}{(\gamma\varepsilon^2)^{\frac{2(\alpha-\beta)}{\alpha}}}\right),\]
which can be proved similarly to the above lemma.
\end{remark}

\section{Proof of Proposition \ref{prop}}
\label{proofprop}
\subsection*{Proof of Proposition \ref{prop}(a)}
\begin{proof}
Recall $F_n$ in \eqref{Fn}. The Markov property of $u(t,\cdot)$ (see Theorem 9.14 on page 248 of \cite{da2014stochastic}) implies
\begin{equation*}
P\left(F_{j}\vert\sigma\{u(t_i,\cdot)\}_{0\leq i<j}\right)=P\left(F_{j}\vert u(t_{j-1},\cdot)\right).
\end{equation*}
If we can show that $P\left(F_{j}\vert u(t_{j-1},\cdot)\right)$ has a uniform bound $\textbf{C}_4\exp\left(-\frac{\textbf{C}_5}{t_1^{\frac{\beta}{\alpha d}}}\right)$ that does not depend on $j$, then the same bound holds for the conditional probability $P\left(F_{j}\Big\vert \bigcap\limits_{k=0}^{j-1}F_{k}\right)$, which is conditioned on a realization of $u(t_k,\cdot),0\leq k<j$. Thus, it is enough to prove
\begin{equation}
\label{F1}
P(F_1)\leq \textbf{C}_4\exp\left(-\frac{\textbf{C}_5}{t_1^{\left(\frac{\beta}{\alpha d}\wedge \frac{\alpha-\beta}{\alpha}\right)}}\right),
\end{equation}
where $\textbf{C}_4$, $\textbf{C}_5$ do not depend on $u_0$ and $\vert u(t,\textbf{x})\vert < t_1^{\frac{\alpha-\beta}{2\alpha}}$ for every $(t,\textbf{x})\in R_{0,n_1-1}$ defined in \eqref{Pij}. 

Now we show how to find a uniform bound for $P(F_1)$ with starting from considering the truncated function
\[f_{t_1}(\textbf{x})=\begin{cases}
\textbf{x}, & \vert \textbf{x}\vert\leq t_1^{\frac{\alpha-\beta}{2\alpha}};\\
\frac{\textbf{x}}{\vert \textbf{x}\vert}\cdot t_1^{\frac{\alpha-\beta}{2\alpha}}, & \vert \textbf{x}\vert> t_1^{\frac{\alpha-\beta}{2\alpha}}.
\end{cases}\]
Particularly, $\vert f_{t_1}(\textbf{x})\vert\leq t_1^{\frac{\alpha-\beta}{2\alpha}}$ and consider the following two equations
\begin{equation}
\begin{split}
\label{shef}
&\partial_tv(t,\textbf{x})=-(-\Delta)^{\alpha/2}v(t,\textbf{x})+\sigma(t,\textbf{x},f_{t_1}(v(t,\textbf{x}))) \dot{F}(t,\textbf{x}),\\
&\partial_tv_g(t,\textbf{x})=-(-\Delta)^{\alpha/2}v_g(t,\textbf{x})+\sigma(t,\textbf{x},f_{t_1}(u_0(\textbf{x}))) \dot{F}(t,\textbf{x})
\end{split}
\end{equation}
with the same initial profile $u_0(\textbf{x})$. Hence, we can decompose $v(t,\textbf{x})$ into
\[v(t,\textbf{x})=v_g(t,\textbf{x})+D(t,\textbf{x}),\]
with
\[D(t,\textbf{x})=\int_{[0,t]\times\mathbb{T}^d} \bar{p}(t-s,\textbf{x}-\textbf{y})[\sigma(s,\textbf{y},f_{t_1}(v(s,\textbf{y})))-\sigma(s,\textbf{y},f_{t_1}(u_0(\textbf{y})))]F(dsd\textbf{y}).\]
The Lipschitz property on the third variable of $\sigma(t,\textbf{x},u)$ in \eqref{hypothesis1} gives
\begin{equation}
\label{sigmadiff1}
\begin{split}
\vert \sigma(s,\textbf{y},f_{t_1}(v(s,\textbf{y})))-\sigma(s,\textbf{y},f_{t_1}(u_0(\textbf{y})))\vert&\leq \mathcal{D}\vert f_{t_1}(v(s,\textbf{y}))-f_{t_1}(u_0(\textbf{y}))\vert\\
&\leq 2\mathcal{D}t_1^{\frac{\alpha-\beta}{2\alpha}}.
\end{split}
\end{equation}

We recall that $R_{i,j}$ in \eqref{Pij} and define a new sequence of events,
\[H_j=\left\lbrace\vert v(t,\textbf{x})\vert\leq t_1^{\frac{\alpha-\beta}{2\alpha}},\forall (t,\textbf{x})\in R_{1,j}\setminus R_{1,j-1}\right\rbrace.\]
Clearly, the property of $f_{t_1}(\textbf{x})$ and \eqref{Fn} imply
\[F_{1}=\bigcap_{j=-n_1+1}^{n_1-1}H_j.\]
Additionally, we define another two sequences of events
\[A_j=\left\lbrace\vert v_g(t,\textbf{x})\vert\leq 2t_1^{\frac{\alpha-\beta}{2\alpha}},\forall (t,\textbf{x})\in R_{1,j}\setminus R_{1,j-1}\right\rbrace\]
and
\[B_j=\left\lbrace\vert D(t,\textbf{x})\vert>t_1^{\frac{\alpha-\beta}{2\alpha}}, \text{~for some~}(t,\textbf{x})\in R_{1,j}\setminus R_{1,j-1}\right\rbrace.\]
It is straightforward to check that
\[H_j^c\supset A_j^c\cap B_j^c,\]
which leads to
\begin{equation}
\label{sumprob}
\begin{split}
P(F_{1})&= P\left(\bigcap_{j=-n_1+1}^{n_1-1}H_j\right)\leq P\left(\bigcap_{j=-n_1+1}^{n_1-1}[A_j\cup B_j]\right)\\
&\leq P\left(\left(\bigcap_{j=-n_1+1}^{n_1-1}A_j\right)\bigcup\left(\bigcup_{j=-n_1+1}^{n_1-1}B_j\right)\right)\\
&\leq P\left(\bigcap_{j=-n_1+1}^{n_1-1}A_j\right)+P\left(\bigcup_{j=-n_1+1}^{n_1-1}B_j\right)\\
&\leq P\left(\bigcap_{j=-n_1+1}^{n_1-1}A_j\right)+\sum_{j=-n_1+1}^{n_1-1} P(B_j),
\end{split}
\end{equation}
where the second inequality can be shown by using induction. 

Moreover, for $j=-n_1+1$,
\begin{equation}
\label{b1}
B_j \subseteq\left\lbrace\sup_{\substack{0\leq s\leq c_0\varepsilon^4\\ \textbf{y}\in \left[(-n_1+1)\varepsilon^2,(-n_1+2)\varepsilon^2\right]^d}}\vert D(s,\textbf{y})\vert>t_1^{\frac{\alpha-\beta}{2\alpha}}\right\rbrace,
\end{equation}
and, for $j>-n_1+1$,
\begin{equation}
\label{bj}
\begin{split}
B_j&\subseteq \left\lbrace\sup_{(t,\textbf{x})\in R_{1,j}\setminus R_{1,j-1}}\vert D(t,\textbf{x})\vert>t_1^{\frac{\alpha-\beta}{2\alpha}}\right\rbrace\\&\subseteq\bigcup_{(t,\textbf{x})\in R_{1,j-1}\setminus R_{1,j-2}}\left\lbrace\sup_{\substack{0\leq s\leq c_0\varepsilon^4\\ \textbf{y}\in \textbf{x}+\left[0,\varepsilon^2\right]^d}}\vert D(s,\textbf{y})\vert>t_1^{\frac{\alpha-\beta}{2\alpha}}\right\rbrace.\\
\end{split}
\end{equation}
From \eqref{sigmadiff1} and Remark \ref{largeremark}, we have
\begin{equation}
\label{bsig}
P\left(\sup_{\substack{0\leq s\leq c_0\varepsilon^4\\ \textbf{y}\in \textbf{x}+\left[0,\varepsilon^2\right]^d}}\vert D(s,\textbf{y})\vert>t_1^{\frac{\alpha-\beta}{2\alpha}}\right)\leq \frac{C_5}{1\wedge \sqrt{{c_0}^d}}\exp\left(-\frac{C_6}{4\mathcal{D}^2t_1^{\frac{\alpha-\beta}{\alpha}}}\right).
\end{equation}
Therefore, \eqref{b1} implies
\begin{equation}
\label{b1upper}
P(B_{-n_1+1})\leq P\left(\sup_{\substack{0\leq s\leq c_0\varepsilon^4\\ \textbf{y}\in \left[0,\varepsilon^2\right]^d}}\vert D(s,\textbf{y})\vert>t_1^{\frac{\alpha-\beta}{2\alpha}}\right),
\end{equation}
and \eqref{bj} implies, for $j>-n_1+1$,
\begin{equation}
\label{bjupper}
P(B_j)\leq \left[(j+n_1-1)^d-(j+n_1-2)^d\right]P\left(\sup_{\substack{0\leq s\leq c_0\varepsilon^4\\ \textbf{y}\in \left[0,\varepsilon^2\right]^d}}\vert D(s,\textbf{y})\vert>t_1^{\frac{\alpha-\beta}{2\alpha}}\right).
\end{equation}
Using \eqref{bsig}, \eqref{b1upper} and \eqref{bjupper}, we therefore obtain that
\begin{equation}
\label{probB}
\begin{split}
\sum_{j=-n_1+1}^{n_1-1}P(B_j)&\leq \left((2n_1-2)^d+1\right)P\left(\sup_{\substack{0\leq s\leq c_0\varepsilon^4\\ \textbf{y}\in \left[0,\varepsilon^2\right]^d}}\vert D(s,\textbf{y})\vert>t_1^{\frac{\alpha-\beta}{2\alpha}}\right)\\
&\leq\frac{C(d)}{\varepsilon^{2d}}\cdot\frac{C_5}{1\wedge \sqrt{{c_0}^d}}\exp\left(-\frac{C_6}{4\mathcal{D}^2t_1^{\frac{\alpha-\beta}{\alpha}}}\right).
\end{split}
\end{equation}

Note that in equation \eqref{shef} for $v_g$, $\sigma$ is deterministic and hence $v_g$ is a Gaussian random field. Then the following lemma provides a lower bound for variance of the noise term $N(t_1,\textbf{x})$ and an upper bound on the decay of covariance between two random variables $N(t_1,\textbf{x})$, $N(t_1,\textbf{y})$ as $\vert \textbf{x}-\textbf{y}\vert$ increases.

\begin{lemma}
\label{varbound}
If we consider different noise terms $N(t_1,\textbf{x})$, $N(t_1,\textbf{y})$ with a deterministic $\sigma(t,\textbf{x},u)=\sigma(t,\textbf{x})$, then there exist constants $C_7,C_8$ and $C_9>0$ depending only on $\mathcal{C}_1$, $\mathcal{C}_2$, $d$, $\alpha$, and $\beta$ such that
\begin{equation*}
C_7t_1^{\frac{\alpha-\beta}{\alpha}}\leq {\rm Var}[N(t_1,\textbf{x})]\leq C_8t_1^{\frac{\alpha-\beta}{\alpha}}
\end{equation*}
and
\begin{equation*}
{\rm Cov}[N(t_1,\textbf{x}),N(t_1,\textbf{y})]\leq C_9t_1\left\vert \textbf{x}-\textbf{y}\right\vert^{-\beta}.
\end{equation*}
\end{lemma}
\begin{proof} We use the Burkholder-Davis-Gundy inequality, assumption of $\sigma$ in \eqref{hypothesis2} and the same argument in \eqref{ppconv} to achieve
\begin{equation}
\begin{split}
\label{varlowbound}
&\textrm{Var}[N(t_1,\textbf{x})]=\mathbb{E}\left[(N(t_1,\textbf{x}))^2\right]\\
&\geq C\mathbb{E}\left\vert \int_0^{t_1}\int_{\mathbb{T}^d}\int_{\mathbb{T}^d}\bar{p}(t_1-s,\textbf{x}-\textbf{y})\bar{p}(t_1-s,\textbf{x}-\textbf{z})\sigma(s,\textbf{y})\sigma(s,\textbf{z})\Lambda(\textbf{y}-\textbf{z})d\textbf{y}d\textbf{z}ds\right\vert\\
&=C\int_0^{t_1}\int_{\mathbb{T}^d}\int_{\mathbb{T}^d}\bar{p}(t_1-s,\textbf{x}-\textbf{y})\bar{p}(t_1-s,\textbf{x}-\textbf{z})\sigma(s,\textbf{y})\sigma(s,\textbf{z})\Lambda(\textbf{y}-\textbf{z})d\textbf{y}d\textbf{z}ds\\
&\geq C\mathcal{C}_1^2\left(\int_0^{t_1}\sum_{\textbf{n}\in\Z^d} \lambda(\textbf{n})\exp\left(-2(2\pi\vert \textbf{n}\vert)^\alpha(t_1-s)\right)ds\right)\\
&\geq C(d,\beta,\mathcal{C}_1)\left(\int_0^{t_1}\int_{\R^d}\vert\textbf{x}\vert^{\beta-d}\exp\left(-2(2\pi\vert \textbf{x}\vert)^\alpha(t_1-s)\right)d\textbf{x}ds\right).
\end{split}
\end{equation}
The last inequality results from \eqref{lambdabound} and the fact that $\exp\left(-2(2\pi\vert \textbf{n}\vert)^\alpha(t_1-s)\right)$ decreases as $\vert \textbf{n}\vert$ increases. Using Fubini's theorem and changing variable with $w=2^{\alpha+1}\pi^\alpha x^\alpha t_1$, we continue \eqref{varlowbound} as follows,
\begin{equation*}
\begin{split}
\textrm{Var}[N(t_1,\textbf{x})]&\geq C(d,\beta,\mathcal{C}_1)\left(\int_0^{t_1}\int_0^\infty x^{\beta-d}\exp\left(-2(2\pi x)^\alpha(t_1-s)\right)x^{d-1}dxds\right)\\
&=C(d,\beta,\mathcal{C}_1)\left(\int_0^\infty\int_0^{t_1} x^{\beta-1}\exp\left(-2(2\pi x)^\alpha(t_1-s)\right)dsdx\right)\\
&=C(d,\beta,\mathcal{C}_1)\int_0^\infty\frac{1-e^{-2^{\alpha+1}\pi^\alpha x^\alpha t_1}}{2^{\alpha+1}\pi^\alpha x^{1+\alpha-\beta}}dx\\
&=C(d,\alpha,\beta,\mathcal{C}_1)t_1^{\frac{\alpha-\beta}{\alpha}}\int_{0}^\infty \frac{1-e^{-w}}{w^{2-\beta/\alpha}}dw,
\end{split}
\end{equation*}
where the last integral converges when $0<\beta<\alpha\wedge d$, which completes the proof of the first part. For the other direction,
\begin{equation*}
\begin{split}
\textrm{Var}[N(t_1,\textbf{x})]&\leq C'\mathbb{E}\left\vert \int_{[0,{t_1}]\times{\mathbb{T}^{2d}}}\bar{p}(t_1-s,\textbf{x}-\textbf{y})\bar{p}(t_1-s,\textbf{x}-\textbf{z})\sigma(s,\textbf{y})\sigma(s,\textbf{z})\Lambda(\textbf{y}-\textbf{z})d\textbf{y}d\textbf{z}ds\right\vert\\
&\leq C'\int_0^{t_1}\int_{\mathbb{T}^d}\int_{\mathbb{T}^d}\bar{p}(t_1-s,\textbf{x}-\textbf{y})\bar{p}(t_1-s,\textbf{x}-\textbf{z})\sigma(s,\textbf{y})\sigma(s,\textbf{z})\Lambda(\textbf{y}-\textbf{z})d\textbf{y}d\textbf{z}ds\\
&\leq C'\mathcal{C}_2^2\left(\int_0^{t_1}\sum_{\textbf{n}\in\Z^d} \lambda(\textbf{n})\exp\left(-2(2\pi\vert \textbf{n}\vert)^\alpha(t_1-s)\right)ds\right)\\
&\leq C'(d,\beta,\mathcal{C}_2)\left(\lambda(\textbf{0})t_1^{\frac{\alpha-\beta}{\alpha}}+\int_0^{t_1}\int_{\R^d}\vert\textbf{x}\vert^{\beta-d}\exp\left(-2(2\pi\vert \textbf{x}\vert)^\alpha(t_1-s)\right)d\textbf{x}ds\right).
\end{split}
\end{equation*}
The last inequality follows from \eqref{lambdabound} and the fact that $t_1<t_1^{\frac{\alpha-\beta}{\alpha}}$ for $t_1<1$. Using the same argument as above, we end up with
\begin{equation*}
{\rm Var}[N(t_1,\textbf{x})]\leq C_8t_1^{\frac{\alpha-\beta}{\alpha}}.
\end{equation*}

Additionally, we apply \eqref{lambdafourier}, \eqref{Lambdabound} and use the fact that $1-e^{-x}\leq x$ to derive an upper bound on the covariance of $N(t_1,\textbf{x})$ and $N(t_1,\textbf{y})$ when $\textbf{x}\neq \textbf{y}$,

\begin{equation*}
\begin{split}
&\textrm{Cov}[N(t_1,\textbf{x}),N(t_1,\textbf{y})]=\E[N(t_1,\textbf{x})N(t_1,\textbf{y})]\\
&\leq \mathcal{C}_2^2\left(\sum_{\textbf{n}\in \Z^d}\lambda(\textbf{n})\exp(2\pi i\textbf{n}\cdot (\textbf{x}-\textbf{y}))\int_0^{t_1}\exp\left(-2(2\pi\vert \textbf{n}\vert)^\alpha(t_1-s)\right)ds\right)\\
&\leq C_9(\mathcal{C}_2)t_1\sum_{\textbf{n}\in \Z^d}\lambda(\textbf{n})\exp(2\pi i\textbf{n}\cdot (\textbf{x}-\textbf{y}))=C_9(\mathcal{C}_2)t_1\Lambda(\textbf{x}-\textbf{y})\leq C_9t_1\left\vert \textbf{x}-\textbf{y}\right\vert^{-\beta},
\end{split}
\end{equation*}
where the last line follows by the definition of $\Lambda(\textbf{x})$ from \eqref{lambdafourier} and \eqref{Lambdabound}.
\end{proof}

We turn our focus on how to compute the upper bound on $P\left(\bigcap\limits_{j=-n_1+1}^{n_1-1}A_j\right)$ in \eqref{sumprob}, so define a sequence of events involving $v_g$, 
\[I_j=\left\lbrace\vert v_g(t,\textbf{x})\vert\leq 2t_1^{\frac{\alpha-\beta}{2\alpha}}, \forall(t,\textbf{x})\in R_{1,j}\right\rbrace\quad\text{and}\quad I_{-n_1}=\Omega.\]
We can write $P\left(\bigcap\limits_{j=-n_1+1}^{n_1-1}A_j\right)$ in terms of a product of conditional probabilities as
\begin{equation}
\label{Acond}
P\left(\bigcap_{j=-n_1+1}^{n_1-1}A_j\right)=P(I_{n_1-1})=P(I_{-n_1})\prod_{j=-n_1+1}^{n_1-1}\frac{P(I_j)}{P(I_{j-1})}=\prod_{j=-n_1+1}^{n_1-1}P(I_j\vert I_{j-1}).
\end{equation}
Let $\mathcal{G}_j$ be the $\sigma-$algebra generated by 
\[N_{(t_1)}(t,\textbf{x}):=\int_{[0,t]\times\mathbb{T}^d} \bar{p}(t-s,\textbf{x}-\textbf{y})\sigma(s,\textbf{y},f_{t_1}(u_0(\textbf{y})))F(dsd\textbf{y}),~~(t,\textbf{x})\in R_{1,j},\]
which is the noise term of $v_g(t_1,\textbf{x})$ in \eqref{shef}. Once we prove that there is a uniform bound on $P\left(I_j\vert \mathcal{G}_{j-1}\right)$, the same bound holds for the conditional probability $P\left(I_j\vert I_{j-1}\right)$. Notice that $\sigma(s,\textbf{y},f_{t_1}(u_0(\textbf{y})))$ is deterministic and uniformly bounded, so by Lemma \ref{varbound}, we have
\begin{equation*}
\textrm{Var}\left[N_{(t_1)}(t_1,\textbf{x})\right]\geq C_7 t_1^{\frac{\alpha-\beta}{\alpha}},
\end{equation*}
and for $(t,\textbf{x})\in R_{1,j}\setminus R_{1,j-1}$, one can decompose
\begin{equation}
\label{vgdecom}
v_g(t,\textbf{x})=\int_{\mathbb{T}^d}\bar{p}(t,\textbf{x}-\textbf{y})u_0(\textbf{y})d\textbf{y}+X+Y,
\end{equation}
where $X=\E\left[N_{(t_1)}(t,\textbf{x})\vert \mathcal{G}_{j-1}\right]$ is a Gaussian random variable, which can be written as
\begin{equation}
\label{Xdecom}
X=\sum_{(t,\textbf{x})\in R_{1,j-1}}\eta^{(j)}(t,\textbf{x})N_{(t_1)}(t,\textbf{x}),
\end{equation}
for some coefficients $\left(\eta^{(j)}(t,\textbf{x})\right)_{(t,\textbf{x})\in R_{1,j-1}}$. Define
\begin{equation*}
Y=N_{(t_1)}(t,\textbf{x})-X,
\end{equation*}
then the conditional variance of $Y$ can be rewritten as
\begin{align*}
&\textrm{Var}(Y\vert\mathcal{G}_{j-1})=\E [(N_{(t_1)}(t,\textbf{x})-X)^2\vert\mathcal{G}_{j-1}]-(\E[N_{(t_1)}(t,\textbf{x})-X\vert\mathcal{G}_{j-1}])^2\\
&=\E[(N_{(t_1)}(t,\textbf{x})-\E[N_{(t_1)}(t,\textbf{x})\vert \mathcal{G}_{j-1}])^2\vert\mathcal{G}_{j-1}]=\textrm{Var}[N_{(t_1)}(t,\textbf{x})\vert \mathcal{G}_{j-1}].
\end{align*}
Since $Y=N_{(t_1)}(t,\textbf{x})-X$ is independent of $\mathcal{G}_{j-1}$, we may express $\textrm{Var}(Y)$ as
\begin{equation*}
\textrm{Var}(Y)=\textrm{Var}(Y\vert\mathcal{G}_{j-1})=\textrm{Var}[N_{(t_1)}(t,\textbf{x})\vert \mathcal{G}_{j-1}].
\end{equation*}

In fact, Anderson's inequality \cite{anderson1955integral} implies that a Gaussian random variable $Z \sim N(\mu,\sigma^2)$ and any $a >0$, the probability $P(\vert Z\vert \leq a)$ is maximized when $\mu = 0$, thus
\begin{equation}
\label{probA}
\begin{split}
P\left(I_j\vert \mathcal{G}_{j-1}\right)&\leq P\left(\vert v_g(t,\textbf{x})\vert\leq 2t_1^{\frac{\alpha-\beta}{2\alpha}}, (t,\textbf{x})\in R_{1,j}\setminus R_{1,j-1}\bigg| \mathcal{G}_{j-1}\right)\\
&\leq P\left(\left\vert Z'\right\vert\leq \frac{2t_1^\frac{\alpha-\beta}{2\alpha}}{\sqrt{\textrm{Var}[N_{(t_1)}(t,\textbf{x})\vert \mathcal{G}_{j-1}]}}\right)
\end{split}
\end{equation}
where $Z'\sim N(0,1)$. Let's use the notation $\textrm{SD}$ to denote the standard deviation of a random variable. By applying the Minkowski inequality to \eqref{vgdecom} and \eqref{Xdecom}, we get
\[\textrm{SD}[N_{(t_1)}(t,\textbf{x})]\leq \textrm{SD}(X)+\textrm{SD}(Y)\]
and
\[\textrm{SD}(X)\leq \sum_{(t,\textbf{x})\in R_{1,j-1}}\left\vert\eta^{(j)}(t,\textbf{x})\right\vert\cdot \textrm{SD}[N_{(t_1)}(t,\textbf{x})].\]

If we can control coefficients by restricting
\[\sum_{(t,\textbf{x})\in R_{1,j-1}}\left\vert\eta^{(j)}(t,\textbf{x})\right\vert<\frac{\sqrt{C_7}}{2\sqrt{C_8}},\]
then by Lemma \ref{varbound}, we obtain
\begin{equation*}
\begin{split}
\textrm{SD}(X)&\leq \sum_{(t,\textbf{x})\in R_{1,j-1}}\left\vert\eta^{(j)}(t,\textbf{x})\right\vert\cdot \textrm{SD}[N_{(t_1)}(t,\textbf{x})]\\
&\leq \left(\sum_{(t,\textbf{x})\in R_{1,j-1}}\left\vert\eta^{(j)}(t,\textbf{x})\right\vert\right)\cdot\sup_{(t,\textbf{x})\in R_{1,j-1}}\textrm{SD}[N_{(t_1)}(t,\textbf{x})]\\
&<\frac{\sqrt{C_7}}{2\sqrt{C_8}}\cdot \sqrt{C_8}t_1^{\frac{\alpha-\beta}{2\alpha}}=\frac{\sqrt{C_7}}{2}t_1^{\frac{\alpha-\beta}{2\alpha}}.
\end{split}
\end{equation*}
Therefore, $\textrm{SD}(Y)$ can be bounded below by
\begin{equation*}
\textrm{SD}(Y)\geq \textrm{SD}(N_{(t_1)}(t,\textbf{x})) - \textrm{SD}(X)>\frac{\sqrt{C_7}}{2}t_1^{\frac{\alpha-\beta}{2\alpha}},
\end{equation*}
so that we can continue to derive the uniform upper bound of $P(I_j\vert \mathcal{G}_{j-1})$ in \eqref{probA} as follows,
\begin{equation*}
\begin{split}
P(I_j\vert \mathcal{G}_{j-1})&\leq P\left(\vert Z'\vert\leq \frac{2t_1^{\frac{\alpha-\beta}{2\alpha}}}{\sqrt{\textrm{Var}[N_{(t_1)}(t,\textbf{x})\vert \mathcal{G}_{j-1}]}}\right)\\
&\leq P\left(\vert Z'\vert\leq \frac{4t_1^{\frac{\alpha-\beta}{2\alpha}}}{\sqrt{C_7 t_1^{\frac{\alpha-\beta}{\alpha}}}}\right)\\
&=P\left(\vert Z'\vert\leq C'\right)<1,
\end{split}
\end{equation*}
where $C'$ depends only on $\mathcal{C}_1$, $d$, $\alpha$, and $\beta$. A bound \eqref{jbound} on $j$ and \eqref{Acond} together yield
\begin{equation}
\label{probAbound}
P\left(\bigcap_{j=-n_1+1}^{n_1-1}A_j\right)\leq C^{\varepsilon^{-2}}=C\exp\left(-\frac{C'}{\varepsilon^2}\right),
\end{equation}
where $C,C'$ depend only on $\mathcal{C}_1$, $d$, $\alpha$, and $\beta$. The following lemma shows how to select $c_0$ to ensure $\sum\limits_{(t,\textbf{x})\in R_{1,j-1}}\vert\eta^{(j)}(t,\textbf{x})\vert\leq \frac{\sqrt{C_7}}{2\sqrt{C_8}}$, which completes the proof.
\end{proof}

\begin{lemma}\label{coeffbound}
For a given $\varepsilon>0$, we may choose $c_0>0$ in \eqref{c0} such that
\begin{equation*}
\sum_{(t,\textbf{x})\in R_{1,j-1}}\vert\eta^{(j)}(t,\textbf{x})\vert\leq \frac{\sqrt{C_7}}{2\sqrt{C_8}}.
\end{equation*}
\end{lemma}
\begin{proof}
Let $X$ and $Y$ be random variables defined in \eqref{vgdecom} and \eqref{Xdecom}. Since $Y$ and $\mathcal{G}_{j-1}$ are independent, for $\forall(t,\textbf{x})\in R_{1,j-1}$, we have
\[\textrm{Cov}[Y,N_{(t_1)}(t,\textbf{x})]=0,\]
and for $(t,\textbf{y})\in R_{1,j}\setminus R_{1,j-1}$,
\begin{equation}
\label{noisecov}
\begin{split}
\textrm{Cov}[N_{(t_1)}(t,\textbf{x}),N_{(t_1)}(t,\textbf{y})]&=\textrm{Cov}[N_{(t_1)}(t,\textbf{x}),X]\\
&=\sum_{(t,\textbf{x'})\in R_{1,j-1}}\eta^{(j)}(t,\textbf{x'}) \textrm{Cov}[N_{(t_1)}(t,\textbf{x}),N_{(t_1)}(t,\textbf{x'})].
\end{split}
\end{equation}

We write the equation \eqref{noisecov} in a matrix form
\begin{equation*}
\textbf{X} =\Sigma\eta,
\end{equation*}
where $\eta=\left(\eta^{(j)}(t,\textbf{x})\right)_{(t,\textbf{x})\in R_{1,j-1}}^T$, $\textbf{X}=\left\lbrace\textrm{Cov}[N_{(t_1)}(t,\textbf{x}),N_{(t_1)}(t,\textbf{y})]\right\rbrace_{(t,\textbf{x})\in R_{1,j-1}}^T,$
and $\Sigma$ is the covariance matrix of $\left(N_{(t_1)}(t,\textbf{x})\right)_{(t,\textbf{x})\in R_{1,j-1}}$. Let $\vert\vert\cdot\vert\vert_{1,1}$ be the matrix norm induced by the $\vert\vert\cdot\vert\vert_{l_1}$ norm, that is for a matrix $\textbf{A}$,
\[\vert\vert\textbf{A}\vert\vert_{1,1}:=\sup_{\textbf{x}\neq \textbf{0}}\frac{\vert\vert \textbf{Ax}\vert\vert_{l_1}}{\vert\vert\textbf{x}\vert\vert_{l_1}}.\]
It can be shown that $\vert\vert \textbf{A}\vert\vert_{1,1}=\max\limits_j\sum\limits_{i=1}^n\vert a_{ij}\vert$ (see page 259 of \cite{rao2000linear}). Therefore, we have
\begin{equation*}
\vert\vert \eta\vert\vert_{l_1}=\vert\vert\Sigma^{-1}\textbf{X}\vert\vert_{l_1}\leq\vert\vert\Sigma^{-1}\vert\vert_{1,1}\vert\vert\textbf{X}\vert\vert_{l_1}.
\end{equation*}

We rewrite $\Sigma=\textbf{D}\textbf{T}\textbf{D}$, where $\textbf{D}$ is a diagonal matrix with entries $\sqrt{\textrm{Var}[N_{(t_1)}(t,\textbf{x})]}$, and $\textbf{T}$ is the correlation matrix with entries
\[e_{\textbf{xx'}}=\frac{\textrm{Cov}[N_{(t_1)}(t,\textbf{x}),N_{(t_1)}(t,\textbf{x'})]}{\sqrt{\textrm{Var}[N_{(t_1)}(t,\textbf{x})]}\cdot\sqrt{\textrm{Var}[N_{(t_1)}(t,\textbf{x'})]}}.\]
Thanks to Lemma \ref{varbound}, for $\textbf{x}\neq \textbf{x'}$, $\vert e_{\textbf{xx'}}\vert$ can be bounded above by
\[\vert e_{\textbf{xx'}}\vert\leq \frac{C_9t_1\left\vert \textbf{x}-\textbf{x'}\right\vert^{-\beta}}{C_7t_1^{1-\beta/\alpha}}.\]

Define $\textbf{A}=\textbf{I}-\textbf{T}$, which has zero diagonal entries, so we can bound $\vert\vert\textbf{A}\vert\vert_{1,1}$ by
\begin{equation}
\label{matrixA}
\begin{split}
\vert\vert\textbf{A}\vert\vert_{1,1}&=\max_{\textbf{x}}\sum_{\textbf{x}\neq\textbf{x'}}\vert e_{\textbf{xx'}}\vert\leq 2 \sum_{(t,\textbf{x})\in R_{1,n_1-1}}\vert e_{\textbf{x0}}\vert=\frac{2C_9t_1^{\beta/\alpha}}{C_7}\sum_{(t,\textbf{x})\in R_{1,n_1-1}}\vert \textbf{x}\vert^{-\beta}\\
&\leq\frac{C(d)C_9t_1^{\beta/\alpha}}{C_7\varepsilon^{2\beta}}\int_0^{\sqrt{d}\varepsilon^{-2}}r^{d-\beta-1}dr=\frac{C(d,\beta)C_9}{C_7}\cdot\frac{(c_0\varepsilon^4)^{\beta/\alpha}}{\varepsilon^{2d}}.
\end{split}
\end{equation}

For any $\varepsilon>0$, $c_0$ is selected to satisfy
\begin{equation*}
\vert\vert\textbf{A}\vert\vert_{1,1}\leq\frac{C(d,\beta)C_9}{C_7}\cdot\frac{(c_0\varepsilon^4)^{\beta/\alpha}}{\varepsilon^{2d}}<\phi<1,
\end{equation*}
for which
\begin{equation}
\label{selectc0}
c_0<C(d,\beta,C_7,C_8,C_9)\phi^{\frac{\alpha}{\beta}}\varepsilon^{\frac{2\alpha d-4\beta}{\beta}}.
\end{equation}
Let's denote $\mathcal{C}=C(d,\beta,C_7,C_8,C_9)\phi^{\frac{\alpha}{\beta}}$, and such $c_0$ will lead to
\begin{equation*}
\vert\vert\textbf{T}^{-1}\vert\vert_{1,1}=\vert\vert(\textbf{I}-\textbf{A})^{-1}\vert\vert_{1,1}\leq \frac{1}{1-\vert\vert\textbf{A}\vert\vert_{1,1}}<\frac{1}{1-\phi},
\end{equation*}
and $\vert\vert\Sigma^{-1}\vert\vert_{1,1}\leq \vert\vert\textbf{D}^{-1}\vert\vert_{1,1}\cdot\vert\vert\textbf{T}^{-1}\vert\vert_{1,1}\cdot\vert\vert\textbf{D}^{-1}\vert\vert_{1,1}\leq \frac{1}{C_7(1-\phi)}t_1^{-\frac{\alpha-\beta}{\alpha}}$. By Lemma \eqref{varbound}, $\vert\vert\textbf{X}\vert\vert_{l_1}$ can be bounded above by
\begin{equation*}
\vert\vert\textbf{X}\vert\vert_{l_1}\leq\sum_{(t,\textbf{x})\in R_{1,j-1}}C_9t_1\left\vert \textbf{x}-\textbf{y}\right\vert^{-\beta},
\end{equation*}
where $(t,\textbf{y})\in R_{1,j}\setminus R_{1,j-1}$. As how we compute the upper bound of $\vert\vert\textbf{A}\vert\vert_{1,1}$ in \eqref{matrixA}, we derive that
\begin{equation*}
\begin{split}
\vert\vert\eta\vert\vert_{l_1}&\leq \frac{1}{C_7(1-\phi)}t_1^{-\frac{\alpha-\beta}{\alpha}}\vert\vert\textbf{X}\vert\vert_{l_1}<\frac{1}{1-\phi}\cdot\sum_{(t,\textbf{x})\in R_{1,j-1}}\frac{C_9t_1\left\vert \textbf{x}-\textbf{y}\right\vert^{-\beta}}{C_7t_1^{\frac{\alpha-\beta}{\alpha}}}\\
&<\frac{\phi}{1-\phi}\rightarrow 0 \quad \text{as } \phi \rightarrow 0.
\end{split}
\end{equation*}
Therefore, we can make 
\begin{equation*}
\sum_{(t,\textbf{x})\in R_{1,j-1}}\vert\eta^{(j)}(t,\textbf{x})\vert\leq \frac{\sqrt{C_7}}{2\sqrt{C_8}}
\end{equation*}
with the selected $c_0$ in \eqref{selectc0}.
\end{proof}

Finally, combining \eqref{sumprob}, \eqref{probB} and \eqref{probAbound} produces
\begin{equation*}
\begin{split}
P(F_{1})&\leq \frac{C(d)C_5}{(1\wedge \sqrt{{c_0}^d})\varepsilon^{2d}}\exp\left(-\frac{C_6}{4\mathcal{D}^2t_1^{\frac{\alpha-\beta}{\alpha}}}\right)+C\exp\left(-\frac{C'}{\varepsilon^2}\right)\\
&\leq C_5'\exp\left(-\frac{d}{2}\ln t_1-\frac{C_6}{4\mathcal{D}^2t_1^{\frac{\alpha-\beta}{\alpha}}}\right)+C\exp\left(-\frac{C'}{\varepsilon^2}\right),
\end{split}
\end{equation*}
and we may choose a constant $\mathcal{D}_1(d,\alpha,\beta)>0$ such that for any $0<\mathcal{D}<\mathcal{D}_1$,
\begin{align*}
P(F_1)&\leq C_5'\exp\left(-\frac{C_6'}{\mathcal{D}^2t_1^{\frac{\alpha-\beta}{\alpha}}}\right)+C\exp\left(-\frac{C'}{\varepsilon^2}\right)\\
&\leq C\exp\left(-\frac{C'}{\varepsilon^2+\mathcal{D}^2t_1^{\frac{\alpha-\beta}{\alpha}}}\right).
\end{align*}
Since $\varepsilon^2=\left(\mathcal{C}_0^{-1}t_1\right)^{\frac{\beta}{\alpha d}}$ from \eqref{c0}, we finally attain
\begin{equation*}
P(F_1)\leq \textbf{C}_4\exp\left(-\frac{\textbf{C}_5}{t_1^{\left(\frac{\beta}{\alpha d}\wedge \frac{\alpha-\beta}{\alpha}\right)}}\right),
\end{equation*}
which completes the proof of \eqref{F1} as well as Proposition \ref{prop}(a).

\subsection*{Proof of Proposition \ref{prop}(b)} 
\begin{proof} The following lemma is crucial to prove Proposition \ref{prop}(b).
\begin{lemma}[The Gaussian correlation inequality]\label{Gaussiancorr}For any convex symmetric sets $K, L$ in $\R^d$ and any centered Gaussian measure $\mu$ on $\R^d$, we have
\[
\mu(K\cap L)\geq \mu(K)\mu(L).
\]
\end{lemma}
\begin{proof}
See in paper \cite{royen2014simple} and \cite{latala2017royen}. 
\end{proof}

By the Markov property of $u(t,\cdot)$, the behavior of $u(t,\cdot)$ in the interval $[t_n,t_{n+1}]$ depends only on $u(t_n,\cdot)$ and $\dot{F}(t,\textbf{x})$ on $[t_n,t]\times\mathbb{T}^d$. Similar to the proof of Proposition \ref{prop}(a), it is enough to show that, when $2\beta\leq\alpha,~d=1$,
\begin{equation}
\label{E0}
P\left(E_{0}\right)\geq  \textbf{C}_6\exp\left(-\frac{\textbf{C}_7}{t_1^{1-\beta/\alpha}}\right),
\end{equation}
where $\textbf{C}_6$, $\textbf{C}_7$ do not depend on $u_0$ and $\vert u_0(\textbf{x})\vert\leq \frac{\mathcal{C}_3}{3}t_1^{\frac{\alpha-\beta}{2\alpha}}$ for every $\textbf{x}\in\mathbb{T}^d$.

We start with the Gaussian case that $\sigma(t,\textbf{x},u)$ is a smooth and deterministic function as $\sigma(t,\textbf{x})$. For $n \geq 0$, we define a sequence of events 
\begin{equation}
\label{Dsequence}
D_n=\left\lbrace\vert u(t_{n+1},\textbf{x})\vert\leq \frac{\mathcal{C}_3}{6}t_1^{\frac{\alpha-\beta}{2\alpha}},\text{and}~\vert u(t,\textbf{x})\vert\leq\frac{2\mathcal{C}_3}{3}t_1^{\frac{\alpha-\beta}{2\alpha}},~\forall t\in[t_n,t_{n+1}],\textbf{x}\in\mathbb{T}^d\right\rbrace.
\end{equation}
We denote 
\[\bar{p}_t(u_0)(\textbf{x})=\bar{p}(t,\cdot)*u_0(\textbf{x})=\int_{\mathbb{T}^d}\bar{p}(t,\textbf{x}-\textbf{y})u_0(\textbf{y})d\textbf{y},\]
and obtain
\begin{equation}
\label{u0convolute}
\vert\bar{p}_t(u_0)(\textbf{x})\vert\leq \sup_{\textbf{x}\in\mathbb{T}^d}\vert u_0(\textbf{x})\vert\leq \frac{\mathcal{C}_3}{3}t_1^{\frac{\alpha-\beta}{2\alpha}}.
\end{equation}

We consider the smooth and bounded function $\frac{\bar{p}_t(u_0)(\textbf{x})}{t_1\sigma(t,\textbf{x})}$ on $[0,t_1]\times\mathbb{T}^d$. \cite{roncal2016fractional} provided a pointwise formula for the fractional Laplacian of a certain class of functions on $\mathbb{T}^d$, so that one may assume that there is a continuous function $f(t,\textbf{x})$ on $[0,t_1]\times\mathbb{T}^d$ such as
\begin{equation}
\label{PtwFracLap}
\begin{split}
f(t,\textbf{x})&=(-\Delta)^{\frac{d-\beta}{2}}\left(\frac{\bar{p}_t(u_0)(\textbf{x})}{t_1\sigma(t,\textbf{x})}\right)\\
&=\int_{\mathbb{T}^d}\left(\frac{\bar{p}_t(u_0)(\textbf{x})}{t_1\sigma(t,\textbf{x})}-\frac{\bar{p}_t(u_0)(\textbf{y})}{t_1\sigma(t,\textbf{y})}\right)K_{d-\beta}(\textbf{x}-\textbf{y})d\textbf{y}\\
&\leq 2\sup_{s,\textbf{x}}\left\vert\frac{\bar{p}_s(u_0)(\textbf{x})}{t_1\sigma(s,\textbf{x})}\right\vert\int_{\mathbb{T}^d}K_{d-\beta}(\textbf{x}-\textbf{y})d\textbf{y},
\end{split}
\end{equation}
where $K_{d-\beta}$ is the positive kernel function on $\mathbb{T}^d$ defined in \cite{roncal2016fractional}, Theorem 1.5.

On the other hand, $f(t,\textbf{x})$ is defined as
\begin{equation*}
\begin{split}
f(t,\textbf{x})& = (-\Delta)^{\frac{d-\beta}{2}}\left(\frac{\bar{p}_t(u_0)(\textbf{x})}{t_1\sigma(t,\textbf{x})}\right)\\
&=\sum_{\textbf{m}\in\mathbb{Z}^d}\vert\textbf{m}\vert^{d-\beta}C_{\textbf{m}}\left(\frac{\bar{p}_t(u_0)(\textbf{x})}{t_1\sigma(t,\textbf{x})}\right)\exp(2\pi i \textbf{m}\cdot\textbf{x}),
\end{split}
\end{equation*}
where $C_{\textbf{m}}\left(\frac{\bar{p}_t(u_0)(\textbf{x})}{t_1\sigma(t,\textbf{x})}\right)$ is the Fourier coefficient. Thus, convolving $f$ with $\Lambda$ on spatial variable yields
\begin{equation}
\label{fconvlambda}
\int_{\mathbb{T}^d}f(t,\textbf{y})\Lambda(\textbf{x}-\textbf{y})d\textbf{y} = \frac{C\bar{p}_t(u_0)(\textbf{x})}{t_1\sigma(t,\textbf{x})}.
\end{equation}

Now for a probability measure $Q$ given by
\[\frac{dQ}{dP}=\exp\left(Z_{t_1}-\frac{1}{2}\langle Z\rangle_{t_1}\right),\]
where
\begin{equation*}
Z_{t_1}=-\int_{[0,t_1]\times\mathbb{T}^d}f(s,\textbf{y})F(dsd\textbf{y}).
\end{equation*}
If $Z_{t_1}$ satisfies Novikov's condition in \cite{allouba1998different}, then for each fixed $T\in[0,\infty)$, $0\leq t\leq T$, and $\forall A\in\mathcal{B}(\mathbb{T}^d)$, 
\begin{equation*}
\widetilde{F}_t(A):=F_t(A)-\int_{[0,t]\times A^{2}}f(s,\textbf{y})\Lambda(\textbf{x}-\textbf{y})d\textbf{x}d\textbf{y}ds
\end{equation*}
is a centered spatially homogeneous Wiener process under the measure $Q$ (see \cite{allouba1998different} for more details). Moreover, for $\textbf{x}\in \mathbb{T}^d$, the covariance structure of $\dot{F}(t,\textbf{x})$ provides the formal form of $\dot{\widetilde{F}}(t,\textbf{x})$ as follows,
\begin{align*}
\dot{\widetilde{F}}(t,\textbf{x})=\dot{F}(t,\textbf{x})+\int_{\mathbb{T}^d}f(t,\textbf{y})\Lambda(\textbf{x}-\textbf{y})d\textbf{y},
\end{align*}
which is a spatially homogeneous noise under measure $Q$. 

\begin{remark}
\label{DimRes}
\cite{roncal2016fractional} proved the pointwise formula when $0<d-\beta<2$, which restricts $d=1,2,3$. 
\end{remark}

We now check whether $Z_{t_1}$ satisfies Novikov's condition. When $f$ is a deterministic function defined in \eqref{PtwFracLap}, it is equivalent to verify that, 
\begin{equation}
\label{novi}
\int_0^{t_1}\int_{\mathbb{T}^d}\int_{\mathbb{T}^d}f(s,\textbf{y})f(s,\textbf{z})\Lambda(\textbf{y}-\textbf{z})d\textbf{z}d\textbf{y}ds<+\infty.
\end{equation}
Given \eqref{u0convolute}, \eqref{PtwFracLap} and \eqref{fconvlambda}, we derive that
\begin{equation}
\label{novikov}
\begin{split}
&\int_0^{t_1}\int_{\mathbb{T}^d}\int_{\mathbb{T}^d}f(s,\textbf{y})f(s,\textbf{z})\Lambda(\textbf{y}-\textbf{z})d\textbf{z}d\textbf{y}ds=\int_0^{t_1}\int_{\mathbb{T}^d}f(s,\textbf{y})\frac{\bar{p}_s(u_0)(\textbf{y})}{t_1\sigma(s,\textbf{y})}d\textbf{y}ds\\
&\leq\int_0^{t_1}\int_{\mathbb{T}^d}\vert f(s,\textbf{y})\vert\left(\frac{\mathcal{C}_3}{3}t_1^{\frac{\alpha-\beta}{2\alpha}}\right)\cdot\frac{1}{t_1\mathcal{C}_1}d\textbf{y}ds=\frac{C(\mathcal{C}_1,\mathcal{C}_3)}{t_1^{\frac{\alpha+\beta}{2\alpha}}}\int_0^{t_1}\int_{\mathbb{T}^d}\vert f(s,\textbf{y})\vert d\textbf{y}ds\\
&\leq\frac{C(d,\beta,\mathcal{C}_1,\mathcal{C}_3)}{t_1^{\frac{\alpha+\beta}{2\alpha}}}\int_0^{t_1}\sup_{s,\textbf{x}}\left\vert\frac{\bar{p}_s(u_0)(\textbf{x})}{t_1\sigma(s,\textbf{x})}\right\vert ds\leq C(d,\beta,\mathcal{C}_1,\mathcal{C}_3)t_1^{-\beta/\alpha}<+\infty,
\end{split}
\end{equation}
which satisfies \eqref{novi}.

Now we can rewrite the equation $(1.1)$ with deterministic $\sigma$ as
\begin{align*}
u(t,\textbf{x})&=\bar{p}_t(u_0)(\textbf{x})+\int_{[0,t]\times\mathbb{T}^d}\bar{p}(t-s,\textbf{x}-\textbf{y})\sigma(s,\textbf{y})\left[\widetilde{F}(dsd\textbf{y})-\frac{\bar{p}_s(u_0)(\textbf{y})}{t_1\sigma(s,\textbf{y})}dsd\textbf{y}\right]\\
&=\bar{p}_t(u_0)(\textbf{x})-\frac{t\bar{p}_t(u_0)(\textbf{x})}{t_1}+\int_{[0,t]\times\mathbb{T}^d}\bar{p}(t-s,\textbf{x}-\textbf{y})\sigma(s,\textbf{y})\widetilde{F}(dsd\textbf{y})\\
&=\left(1-\frac{t}{t_1}\right)\bar{p}_t(u_0)(\textbf{x})+\int_{[0,t]\times\mathbb{T}^d}\bar{p}(t-s,\textbf{x}-\textbf{y})\sigma(s,\textbf{y})\widetilde{F}(dsd\textbf{y}).
\end{align*}
The first term is $0$ at time $t_1$, and the assumption $\vert u_0(\textbf{x})\vert\leq \frac{\mathcal{C}_3}{3}t_1^{\frac{\alpha-\beta}{2\alpha}}$ leads to, for any $(t,\textbf{x})\in[0,t_1)\times\mathbb{T}^d$,
\begin{equation}
\label{initialbound}
\left|\left(1-\frac{t}{t_1}\right)\bar{p}_t(u_0)(\textbf{x})\right|\leq \frac{\mathcal{C}_3}{3}t_1^{\frac{\alpha-\beta}{2\alpha}}.
\end{equation}

We define
\[\widetilde{N}(t,\textbf{x}):=\int_{[0,t]\times\mathbb{T}^d}\bar{p}(t-s,\textbf{x}-\textbf{y})\sigma(s,\textbf{y})\widetilde{F}(d\textbf{y}ds).\]
$\widetilde{F}$ may not have the same covariance as $F$ under the probability measure $P$; however, under probability measure $Q$, we can apply Lemma \ref{larged} to $\widetilde{F}$ and get
\begin{align*}
Q\left(\sup_{\substack{0\leq t\leq c_0\varepsilon^4\\ \textbf{x}\in[0,c_0\varepsilon^2]^d}}\vert \widetilde{N}(t,\textbf{x})\vert>\frac{\mathcal{C}_3}{6}t_1^{\frac{\alpha-\beta}{2\alpha}}\right)&=Q\left(\sup_{\substack{0\leq t\leq (c_0)^{-1}(\sqrt{c_0}\varepsilon)^4\\ \textbf{x}\in[0,(\sqrt{c_0}\varepsilon)^2]^d}}\vert \widetilde{N}(t,\textbf{x})\vert>\frac{\mathcal{C}_3}{6}t_1^{\frac{\alpha-\beta}{2\alpha}}\right)\\
&\leq C_5\exp\left(-\frac{\mathcal{C}_3^2C_6}{36\mathcal{C}_2^2}\right)<1,
\end{align*}
where $\gamma = c_0^{-1}>1$ and $\kappa=\frac{\mathcal{C}_3c_0^{(\beta-\alpha)/2\alpha}}{6}$ in Lemma \ref{larged}. The last inequality is derived from the definition of $\mathcal{C}_3$ in \eqref{c3}. 

\begin{remark}
\label{largebeta}
When $d=1$ and $\beta<\alpha<2\beta$, from \eqref{c0}, $c_0>1$ as $\varepsilon$ approaching $0$, thus we would rather consider another inequality for the probability under measure $Q$,
\begin{equation*}
Q\left(\sup_{\substack{0\leq t\leq c_0\varepsilon^4\\ \textbf{x}\in[0,\varepsilon^2]^d}}\vert \widetilde{N}(t,\textbf{x})\vert>\frac{\mathcal{C}_3}{6}t_1^{\frac{\alpha-\beta}{2\alpha}}\right)\leq C_5\exp\left(-\frac{\mathcal{C}_3^2C_6}{36\mathcal{C}_2^2}\right)<1,
\end{equation*}
where $\gamma = c_0>1$ and $\kappa=\frac{\mathcal{C}_3c_0^{(\alpha-\beta)/2\alpha}}{6}$ in Lemma \ref{larged}.
\end{remark}

By the Gaussian correlation inequality (Lemma \ref{Gaussiancorr}), we obtain
\begin{equation}
\label{Q}
\begin{split}
Q\left(\sup_{\substack{0\leq t\leq t_1\\ \textbf{x}\in\mathbb{T}^d}}\vert \widetilde{N}(t,\textbf{x})\vert\leq\frac{\mathcal{C}_3}{6}t_1^{\frac{\alpha-\beta}{2\alpha}}\right)&\geq Q\left(\sup_{\substack{0\leq t\leq t_1\\ \textbf{x}\in[0,c_0\varepsilon^2]^d}}\vert \widetilde{N}(t,\textbf{x})\vert\leq\frac{\mathcal{C}_3}{6}t_1^{\frac{\alpha-\beta}{2\alpha}}\right)^{\left(\frac{1}{c_0\varepsilon^2}\right)^d}\\
&\geq \left[1-C_5\exp\left(-\frac{\mathcal{C}_3^2C_6}{36\mathcal{C}_2^2}\right)\right]^{\left(\frac{1}{c_0\varepsilon^2}\right)^d}.
\end{split}
\end{equation}
From \eqref{Dsequence} and \eqref{initialbound}, we have the following inequality under measure $Q$,
\begin{equation}
\label{D0}
Q(D_0)\geq Q\left(\sup_{\substack{0\leq t\leq t_1\\ \textbf{x}\in\mathbb{T}^d}}\vert \widetilde{N}(t,\textbf{x})\vert\leq\frac{\mathcal{C}_3}{6}t_1^{\frac{\alpha-\beta}{2\alpha}}\right).
\end{equation}

Since $\frac{dQ}{dP}$ is a Radon-Nikodym derivative,
\begin{equation}
\label{radon}
1=\E\left[\frac{dQ}{dP}\right]=\E\left[\exp\left(Z_{t_1}-\frac{1}{2}\langle Z\rangle_{t_1}\right)\right]=\E[\exp\left(2Z_{t_1}-2\langle Z\rangle_{t_1}\right)],
\end{equation}
and if we replace $f(s,\textbf{y})$ with $2f(s,\textbf{y})$ in $Z_{t_1}$, given that $f(s,\textbf{y})$ is deterministic, we may estimate the Radon-Nikodym derivative in the following way,
\begin{equation}
\label{radonnikodym}
\begin{split}
\E\left[\left(\frac{dQ}{dP}\right)^2\right]&=\E[\exp\left(2Z_{t_1}-\langle Z\rangle_{t_1}\right)]=\E[\exp\left(2Z_{t_1}-2\langle Z\rangle_{t_1}\right)\cdot\exp(\langle Z\rangle_{t_1})]\\
&\leq \exp\left(C(d,\beta,\mathcal{C}_1,\mathcal{C}_3)t_1^{-\beta/\alpha}\right).
\end{split}
\end{equation}
The last inequality comes from \eqref{novikov} and \eqref{radon}, and the Cauchy-Schwarz inequality implies
\[Q(D_0)\leq \sqrt{\E\left[\left(\frac{dQ}{dP}\right)^2\right]}\cdot\sqrt{P(D_0)}.\]
As a consequence of \eqref{Q}, \eqref{D0} and \eqref{radonnikodym}, we conclude that
\begin{equation}
\label{probD}
\begin{split}
P(D_0)&\geq \exp\left(-C(d,\beta,\mathcal{C}_1,\mathcal{C}_3)t_1^{-\beta/\alpha}\right)\exp\left(\frac{C'}{c_0^d\varepsilon^{2d}}\ln\left[1-C_5\exp\left(-\frac{\mathcal{C}_3^2C_6}{36\mathcal{C}_2^2}\right)\right]\right)\\
&\geq \exp\left(-C(d,\beta,\mathcal{C}_1,\mathcal{C}_3)t_1^{-\beta/\alpha}-C'(c_0\varepsilon^2)^{-d}\right)\\
&\geq C\exp\left(-C't_1^{\beta/\alpha-d}\right),
\end{split}
\end{equation}
Recall from \eqref{c0} that $c_0^d\varepsilon^{2d}=\mathcal{C}_0^{\beta/\alpha}t_1^{d-\beta/\alpha}$, so the term $c_0^{-d}\varepsilon^{-2d}$ will dominate the term $t_1^{-\beta/\alpha}$ as $\varepsilon$ approaching 0, which leads to the last inequality.

\begin{remark}
\label{largebeta2}
When $d=1$ and $\beta<\alpha<2\beta$, Remark \ref{largebeta} yields the following inequality,
\begin{equation}
\label{probD2}
\begin{split}
P(D_0)&\geq \exp\left(-C(d,\beta,\mathcal{C}_1,\mathcal{C}_3)t_1^{-\beta/\alpha}\right)\exp\left(\frac{C'}{\varepsilon^{2d}}\ln\left[1-C_5\exp\left(-\frac{\mathcal{C}_3^2C_6}{36\mathcal{C}_2^2}\right)\right]\right)\\
&\geq C\exp\left(-C't_1^{-\beta/\alpha}\right).
\end{split}
\end{equation}
\end{remark}

We use the freezing technique to deal with the case that $\sigma(t,\textbf{x},u)$ is non-deterministic. Define
\[u(t,\textbf{x})=u_g(t,\textbf{x})+M(t,\textbf{x}),\]
where $u_g(t,\textbf{x})$ satisfies the equation
\[\partial_t u_g(t,\textbf{x})=-(-\Delta)^{\alpha/2}u_g(t,\textbf{x})+\sigma(t,\textbf{x},u_0(\textbf{x}))\dot{F}(t,\textbf{x})\]
with an initial profile $u_0$, and 
\[M(t,\textbf{x})=\int_{[0,t]\times\mathbb{T}^d}\bar{p}(t-s,\textbf{x}-\textbf{y})[\sigma(s,\textbf{y},u(s,\textbf{y}))-\sigma(s,\textbf{y},u_0(\textbf{y}))]F(dsd\textbf{y}).\]

Clearly $u_g$ is Gaussian and $\vert u_0(\textbf{x})\vert\leq \frac{\mathcal{C}_3}{3}t_1^{\frac{\alpha-\beta}{2\alpha}}$, so that, for an event defined as
\[\widetilde{D}_0=\left\lbrace\vert u_g(t_{1},\textbf{x})\vert\leq \frac{\mathcal{C}_3}{6}t_1^\frac{\alpha-\beta}{2\alpha},\text{and}~\vert u_g(t,\textbf{x})\vert\leq\frac{2\mathcal{C}_3}{3}t_1^\frac{\alpha-\beta}{2\alpha}~\forall t\in[0,t_{1}],\textbf{x}\in\mathbb{T}^d\right\rbrace,\]
we can apply \eqref{probD} and \eqref{probD2} to acquire
\begin{equation}
\label{probD0}
P\left(\widetilde{D}_0\right)\geq\begin{cases}
C\exp\left(-C't_1^{-\beta/\alpha}\right) & \beta<\alpha<2\beta,~d=1;\\
C\exp\left(-C't_1^{\beta/\alpha-d}\right) & \text{otherwise}.
\end{cases}
\end{equation}

We define the stopping time 
\[\tau=\inf\left\lbrace t>0:\vert u(t,\textbf{x})-u_0(\textbf{x})\vert>2\mathcal{C}_3t_1^\frac{\alpha-\beta}{2\alpha}\text{~for some $\textbf{x}\in\mathbb{T}^d$}\right\rbrace,\]
and if the set is empty, we set $\tau=+\infty$. Clearly we have $\tau>t_1$ on the event $E_{0}$ defined in \eqref{En} and $\vert u(t,\textbf{x})\vert\leq \mathcal{C}_3t_1^\frac{\alpha-\beta}{2\alpha}$ for $\forall t\in[0,t_1]$ on the event $E_{0}$. We make another definition
\[\widetilde{M}(t,\textbf{x})=\int_{[0,t]\times\mathbb{T}^d}\bar{p}(t-s,\textbf{x}-\textbf{y})[\sigma(s,\textbf{y},u(s\wedge \tau,\textbf{y}))-\sigma(s,\textbf{y},u_0(\textbf{y}))]F(dsd\textbf{y}),\]
and $M(t,\textbf{x})=\widetilde{M}(t,\textbf{x})$ for $t\leq t_1$ on the event $\{\tau>t_1\}$. Moreover, we have
\begin{equation}
\label{E0bound}
\begin{split}
P(E_{0})&\geq P\left(\widetilde{D}_0\bigcap \left\lbrace \sup_{\substack{0\leq t\leq t_1\\ \textbf{x}\in\mathbb{T}^d}}\vert M(t,\textbf{x})\vert\leq\frac{\mathcal{C}_3}{6}t_1^\frac{\alpha-\beta}{2\alpha}\right\rbrace\right)\\
&=P\left(\left(\widetilde{D}_0\bigcap \left\lbrace \sup_{\substack{0\leq t\leq t_1\\ \textbf{x}\in\mathbb{T}^d}}\vert M(t,\textbf{x})\vert\leq\frac{\mathcal{C}_3}{6}t_1^\frac{\alpha-\beta}{2\alpha}\right\rbrace\bigcap\{\tau>t_1\}\right)\right.\\
&\hspace{3cm}\bigcup\left.\left(\widetilde{D}_0\bigcap \left\lbrace \sup_{\substack{0\leq t\leq t_1\\ \textbf{x}\in\mathbb{T}^d}}\vert M(t,\textbf{x})\vert\leq\frac{\mathcal{C}_3}{6}t_1^\frac{\alpha-\beta}{2\alpha}\right\rbrace\bigcap\{\tau\leq t_1\}\right)\right).
\end{split}
\end{equation}
It is easy to check that, on the event $\{\tau>t_1\}$,
\[\sup_{\substack{0\leq t\leq t_1\\ \textbf{x}\in\mathbb{T}^d}}\vert M(t,\textbf{x})\vert=\sup_{\substack{0\leq t\leq t_1\\ \textbf{x}\in\mathbb{T}^d}}\vert \widetilde{M}(t,\textbf{x})\vert,\]
and on the event $\widetilde{D}_0\cap\{\tau\leq t_1\}$,
\[\vert u_g(\tau,\textbf{x})\vert\leq \frac{2\mathcal{C}_3}{3}t_1^\frac{\alpha-\beta}{2\alpha} ~\text{and}~ \vert u_0(\textbf{x})\vert\leq\frac{\mathcal{C}_3}{3}t_1^\frac{\alpha-\beta}{2\alpha} ~\text{for all}~ \textbf{x}\in\mathbb{T}^d,\]
\[~\vert u(\tau,\textbf{x})-u_0(\textbf{x})\vert>2\mathcal{C}_3t_1^\frac{\alpha-\beta}{2\alpha} ~\text{for some}~ \textbf{x}\in\mathbb{T}^d.\]
The above inequalities lead to
\begin{align*}
\sup_{\substack{\textbf{x}}}\vert M(\tau,\textbf{x})\vert&=\sup_{\substack{\textbf{x}}}\vert u(\tau,\textbf{x})-u_g(\tau,\textbf{x})\vert\geq \sup_{\substack{\textbf{x}}}(\vert u(\tau,\textbf{x})\vert-\vert u_g(\tau,\textbf{x})\vert)\\
&\geq \sup_{\substack{\textbf{x}}}\vert u(\tau,\textbf{x})\vert-\frac{2\mathcal{C}_3}{3}t_1^\frac{\alpha-\beta}{2\alpha}\geq 2\mathcal{C}_3t_1^\frac{\alpha-\beta}{2\alpha}-\frac{\mathcal{C}_3}{3}t_1^\frac{\alpha-\beta}{2\alpha}-\frac{2\mathcal{C}_3}{3}t_1^\frac{\alpha-\beta}{2\alpha}\\
&>\frac{\mathcal{C}_3}{6}t_1^\frac{\alpha-\beta}{2\alpha},
\end{align*}
which implies
\[\widetilde{D}_0\cap \left\lbrace \sup_{\substack{0\leq t\leq t_1\\ \textbf{x}\in\mathbb{T}^d}}\vert M(t,\textbf{x})\vert\leq\frac{\mathcal{C}_3}{6}t_1^\frac{\alpha-\beta}{2\alpha}\right\rbrace\cap\{\tau\leq t_1\}=\emptyset.\]
Combining above arguments with \eqref{E0bound} yields
\begin{equation}
\label{E0boundsim}
\begin{split}
P(E_{0})&\geq P\left(\widetilde{D}_0\bigcap \left\lbrace \sup_{\substack{0\leq t\leq t_1\\ \textbf{x}\in\mathbb{T}^d}}\vert \widetilde{M}(t,\textbf{x})\vert\leq\frac{\mathcal{C}_3}{6}t_1^\frac{\alpha-\beta}{2\alpha}\right\rbrace\bigcap\{\tau>t_1\}\right)\\
&\geq P(\widetilde{D}_0)-P\left( \sup_{\substack{0\leq t\leq t_1\\ \textbf{x}\in\mathbb{T}^d}}\vert \widetilde{M}(t,\textbf{x})\vert>\frac{\mathcal{C}_3}{6}t_1^\frac{\alpha-\beta}{2\alpha}\right),
\end{split}
\end{equation}
when $\vert u(t,\textbf{x})-u_0(\textbf{x})\vert\leq 2\mathcal{C}_3t_1^\frac{\alpha-\beta}{2\alpha}$ for all $t\in[0,t_1]$ and $\textbf{x}\in\mathbb{T}^d$. Applying Remark \ref{largeremark} to $\widetilde{M}(t,\textbf{x})$ and a union bound method similar to \eqref{probB} establish
\begin{equation}
\label{probDtilde}
P\left( \sup_{\substack{0\leq t\leq t_1\\ \textbf{x}\in\mathbb{T}^d}}\vert \widetilde{M}(t,\textbf{x})\vert>\frac{\mathcal{C}_3}{6}t_1^\frac{\alpha-\beta}{2\alpha}\right)\leq \frac{C_5}{(c_0\varepsilon^4)^{d/2}}\exp\left(-\frac{C_6}{144\mathcal{D}^2 t_1^{\frac{\alpha-\beta}{\alpha}}}\right).
\end{equation}
Consequently, from \eqref{probD0}, \eqref{E0boundsim}, \eqref{probDtilde}, we conclude that there exists a constant $\mathcal{D}_2(d,\alpha,\beta)>0$ such that for any $\mathcal{D}<\mathcal{D}_2$,
\begin{enumerate}
\item[(a)] when $2\beta\leq\alpha,~d=1$,
\begin{equation*}
\begin{split}
P(E_{0})&\geq C\exp\left(-C't_1^{\beta/\alpha-1}\right)-C_5'\exp\left(-\frac{C_6'}{t_1^{\frac{\alpha-\beta}{\alpha}}}\right)\\
&\geq \textbf{C}_6\exp\left(-\frac{\textbf{C}_7}{t_1^\frac{\alpha-\beta}{\alpha}}\right);
\end{split}
\end{equation*}

\item[(b)] when $\beta<\alpha<2\beta,~d=1$,
\begin{equation}
\label{E0a}
P(E_{0})\geq C\exp\left(-C't_1^{-\beta/\alpha}\right)-C_5'\exp\left(-\frac{C_6'}{t_1^{\frac{\alpha-\beta}{\alpha}}}\right),
\end{equation}
however, we could not get a positive lower bound for $P(E_{0})$ due to the fact that $t_1^{-\beta/\alpha} > t_1^{\beta/\alpha-1}$ when $\varepsilon$ is small;

\item[(c)] when $d>1$,
\begin{equation}
\label{E0b}
P(E_{0})\geq C\exp\left(-C't_1^{\beta/\alpha-d}\right)-C_5'\exp\left(-\frac{C_6'}{t_1^{\frac{\alpha-\beta}{\alpha}}}\right),
\end{equation}
however, we could not get a positive lower bound for $P(E_{0})$ due to the fact that $t_1^{\beta/\alpha-d} > t_1^{\beta/\alpha-1}$ when $\varepsilon$ is small,
\end{enumerate}
which completes the proof of \eqref{E0} as well as Proposition \ref{prop}(b).
\end{proof}




\begin{acks}
The author would like to thank his advisor Professor Carl Mueller and the anonymous referees for constructive comments.
\end{acks}


\end{document}